\documentclass[12pt,a4paper,leqno]{amsart}

\usepackage[latin1]{inputenc}
\usepackage[T1]{fontenc}
\usepackage{amsfonts}
\usepackage{amsmath}
\usepackage{amssymb}
\usepackage{eurosym}
\usepackage{mathrsfs}
\usepackage{palatino}

\newcommand{\R}{\mathbb{R}}
\newcommand{\C}{\mathbb{C}}

\newcommand{\N}{\mathbb{N}}

\newcommand{\Z}{\mathbb{Z}}

\numberwithin{equation}{section}

\makeatletter
  \let\c@subsection\c@equation
\makeatother

\swapnumbers
\theoremstyle{plain}
\newtheorem{thm}[equation]{Theorem}
\newtheorem{lem}[equation]{Lemma}
\newtheorem{prop}[equation]{Proposition}

\theoremstyle{definition}

\theoremstyle{remark}
\newtheorem{rem}[equation]{Remark}

\numberwithin{equation}{section}

\title[Non-homogeneous $Tb$ theorem on metric spaces]{Non-homogeneous $Tb$ theorem and random dyadic cubes on metric measure spaces}

\author{Tuomas Hytönen \and Henri Martikainen}\thanks{The authors are supported by the Academy of Finland through the project ``$L^p$ methods in harmonic analysis''.}
\address{Department of Mathematics and Statistics, University of Helsinki, P.O.B. 68, FI-00014 Helsinki, Finland}
\email{tuomas.hytonen@helsinki.fi, henri.martikainen@helsinki.fi}


\begin{document}

\maketitle

\begin{abstract}
We prove a $Tb$ theorem on quasimetric spaces equipped with what we call an upper doubling measure.
This is a property that encompasses both the doubling measures and those satisfying the upper power bound $\mu(B(x,r)) \le Cr^d$. Our spaces are only assumed to satisfy the geometric doubling property: every ball of radius $r$ can be covered by at most $N$ balls of radius $r/2$. A key ingredient is the construction of random systems of dyadic cubes in such spaces.
\end{abstract}

\section{Introduction}

In the introduction to their celebrated paper on the non-homogeneous $Tb$ theorem \cite[p. 153]{NTV}, Nazarov, Treil and Volberg point out that ``a (more or less) complete theory of Calder\'on--Zygmund operators on non-homogeneous spaces -- -- can be developed in an abstract metric space with measure.'' Although a number of results for non-homogeneous singular integral operators (such as the weak-type $L^1$ inequality under a priori $L^2$ boundedness \cite{NTV:weak}, estimates on Lipschitz spaces \cite{GCG:CZO}, and a certain restricted version of the $T1$ theorem \cite{Bramanti}---essentially with $T1,T^*1\in L^{\infty}$, and only on bounded spaces) have been established in quite general metric measure spaces, it seems that the quoted remark has not been fully elaborated for the deeper aspects of the theory. The goal of this paper is to close this gap, and in fact obtain a new level of generality even in the context of $\R^n$. 

Let us describe our setting in more detail. We consider quasimetric spaces with the following well-established postulate, which we refer to as geometric doubling: every ball of radius $r$ can be covered by at most $N$ balls of radius $r/2$. This is essentially the original definition of ``a space of homogeneous nature'' by Coifman and Weiss \cite[pp.~66--67]{CW}, although this name is now commonly used for quasimetric spaces equipped with a doubling measure, the ``particularly important case'' pointed out by Coifman and Weiss immediately after their general definition. It is known that if a metric space is geometrically doubling and complete, then it also supports some doubling measures \cite{LS}; however, we do not assume completeness, and we regard our measure of interest as given by the problem at hand, and not as something that one is free to choose or construct.

We consider so-called upper doubling measures $\mu$, introduced in \cite{Hy}, which constitute a simultaneous generalization of doubling measures and those with the upper power bound property $\mu(B(x,r)) \leq Cr^d$, which are the ones usually considered in the literature on non-homogeneous analysis. But note that power bounded measures are only different, not more general than, the doubling measures. While the original motivation behind the notion of upper doubling in \cite{Hy} was to find a natural unified framework for the doubling and power bounded theories, it also encapsulates other examples of interest. Indeed, although it was not our specific goal, this general framework allows to essentially recapture the recent $T1$ theorem of Volberg and Wick \cite{VW} for ``Bergman-type'' operators.

We now discuss the general strategy and some new aspects of the proof. We follow the basic approach from \cite{NTV} and try to adapt the treatment of their most general cases into our situation. First, the assumption that $\mu$ is merely upper doubling causes for example the effect that the bounds for $\mu(B(x,r))$ depend not only on $r$ but also on $x$. We formulate the kernel estimates in a natural way adapted to this,
and carry out all the estimates with this extra generality. Second, the fact that we work in abstract quasimetric spaces complicates many things. However, note that parts of the relevant BMO and RBMO aspects of this theory were already dealt with in \cite{Hy}. 

A key ingredient behind the proof of Nazarov et al. \cite{NTV} is the random choice of a system of dyadic cubes, so that certain ``bad'' situations can be handled by arguing that their occurrence has only a small probability. It is then clear that we need something similar in an abstract quasimetric space. For a system of dyadic cubes as such, there is a well-known construction due to Christ \cite{Ch}, which serves as our starting point. Even then, it is not obvious how to choose a random system, since the randomization procedure of Nazarov et al. heavily relies on the action of the translation group on $\R^n$. Our solution to this problem, which is based on randomly choosing new ``centers'' for the dyadic cubes of generation $k$ among the old centers of the smaller cubes of generation $k+1$, appears to be new, and it may also be of independent interest, besides the present application to the $Tb$ theorem. In this respect, we note that a family of dyadic systems in a quasi-metric space, rather than just a fixed one, was already exploited by Verbitsky and Wheeden \cite{VeWh} in the context of weighted norm inequalities, but they required the underlying space to have a group structure, so that the new dyadic systems could still be obtained by simply shifting a given one, in analogy to the Euclidean setting.

We shall also employ a closely related construction of random almost-coverings of the space by balls of comparable radius, by which we mean that any given point has a small probability of not being covered. The need for this is related to the fact that, unlike in $\R^n$, it now seems far more natural to formulate notions like the weak boundedness property and the BMO space in terms of balls rather than cubes, and so we essentially need to cut our dyadic cubes into comparable balls when estimating the ``diagonal'' part of the operator.

In the following section, we give detailed statements of the results discussed here, which are then proven in the rest of the paper. In the final section, we describe the relation to the above-mentioned results of Volberg and Wick \cite{VW}.

\section{Preliminaries and the main result}
We now give the detailed definitions, fix some notations and parameters, and then formulate our main theorem, Theorem \ref{maintheorem}.

\subsection{Geometrically doubling regular quasimetric spaces}
Recall that a quasimetric is almost like a metric but the triangle inequality is replaced by the requirement that for some $A_0 \ge 1$ it holds
$\rho(x,y) \le A_0(\rho(x,z) + \rho(z,y))$ for all $x,y,z \in X$. 
A quasimetric space $(X, \rho)$ is geometrically doubling if every open ball $B(x,r) = \{y \in X: \rho(y,x) < r\}$ can be covered by at most $N$ balls of radius $r/2$.
We use this somewhat non-standard name to clearly differentiate this property from other types of doubling properties. We adapt the convention that a ball $B$ is equipped with a fixed
center $c_B \in X$ and radius $r_B > 0$ (if no other notation is at place, we use this). Also, we set $n = \log_2 N$, which can be viewed as (an upper bound for) a geometric dimension of the space. Let us state the following well-known lemma.

\begin{lem}
In a geometrically doubling quasimetric space, a ball $B(x,r)$ can contain the centers $x_i$ of at most $N\alpha^{-n}$ disjoint balls $B(x_i, \alpha r)$ for $\alpha \in (0,1]$.
\end{lem}

Note also that there is a uniform constant depending only on $N$ and $A_0$ so that all subsets of $X$ are geometrically doubling with this constant. Choosing $N$ large enough in the first place, let us use the same constant $N$ everywhere.

For many purposes, quasimetrics are just as good as metrics, only somewhat more annoying to deal with due to the presence of the additional constant in the triangle inequality. However, for some of the more delicate estimates, it seems to us that general quasimetrics can be a bit too wild, and we always ask that our quasimetric $\rho$ satisfy the following regularity property: for every $\epsilon > 0$ there exists
$A(\epsilon) < \infty$ so that
\begin{displaymath}
  \rho(x,y) \le (1+\epsilon)\rho(x,z) + A(\epsilon)\rho(z, y).
\end{displaymath}
Notice that this property is in particular satisfied by all positive powers of an honest metric, and every quasimetric is equivalent to one of that form by a well-known  result of Mac\'ias and Segovia \cite{MS}. While it is easy to cook up irregular quasimetrics, it seems that practically all reasonable examples of quasimetrics from applications already satisfy the regularity property even without passing to an equivalent version. This is in particular the case for all the examples of quasi-metrics pointed out by Coifman and Weiss \cite[p.~68]{CW}.

Much of our subsequent assumptions will be essentially invariant under the change to an equivalent quasimetric, which we explicitly exploit through the mentioned result of Mac\'ias and Segovia, so that a large part of the proof can be carried out in an honest metric space. However, we want to use indicators of balls (of the given quasimetric) as test functions, and it is here that the general quasimetric balls seem to be somewhat too arbitrary for our purposes.

\subsection{Upper doubling measures}

A Borel measure $\mu$ in a quasimetric space $(X, \rho)$ is called upper doubling if there exists a dominating function $\lambda\colon X \times (0, \infty) \to (0,\infty)$ so that
$r \mapsto \lambda(x,r)$ is non-decreasing, $\lambda(x,2r) \le C_{\lambda}\lambda(x,r)$ and $\mu(B(x,r)) \le \lambda(x,r)$ for all $x \in X$ and $r > 0$. The number $d:=\log_2 C_{\lambda}$ can be thought of as (an upper bound for) a dimension of the measure $\mu$, and it will play a similar role as the quantity denoted by the same symbol in \cite{NTV}.

\begin{lem}\label{baest}
We have for every ball $B = B(c_B, r_B)$ and for every $\epsilon > 0$ that
\begin{displaymath}
\int_{X \setminus B} \frac{\rho(x,c_B)^{-\epsilon}}{\lambda(c_B, \rho(x,c_B))}\,d\mu(x) \le C_{\lambda}A_{\epsilon}r_B^{-\epsilon},
\end{displaymath}
where $A_{\epsilon} = 2^{\epsilon}/(2^{\epsilon}-1)$.
\end{lem}

\begin{proof}
We calculate
\begin{eqnarray*}
\int_{X \setminus B} \frac{\rho(x,c_B)^{-\epsilon}}{\lambda(c_B, \rho(x,c_B))}\,d\mu(x) & = & \sum_{j=0}^{\infty} \int_{2^jr_B \le \rho(x,c_B) < 2^{j+1}r_B} \frac{\rho(x,c_B)^{-\epsilon}}{\lambda(c_B, \rho(x,c_B))}\,d\mu(x) \\
& \le & \sum_{j=0}^{\infty} \frac{\mu(B(c_B, 2^{j+1}r_B))}{\lambda(c_B, 2^jr_B)}(2^jr_B)^{-\epsilon} \\
& \le & C_{\lambda}r_B^{-\epsilon}\sum_{j=0}^{\infty} 2^{-\epsilon j} 
 = C_{\lambda}A_{\epsilon}r_B^{-\epsilon},
\end{eqnarray*}
where we used that $\lambda$ is non-decreasing and $\mu(B(c_B, 2^{j+1}r_B)) \le \lambda(c_B, 2^{j+1}r_B) \le C_{\lambda}\lambda(c_B, 2^jr_B)$.
\end{proof}

In what follows, we work in a geometrically doubling  regular quasimetric space $(X,\rho)$ (with the constants $N$ and $n$ as above, and the related function $A(\epsilon)$), which is equipped with an upper doubling measure $\mu$ with the related majorant~$\lambda$.

\subsection{Standard kernels and Calder\'on--Zygmund operators}

Define $\Delta = \{(x,x): x \in X\}$. A standard kernel
is a mapping $K\colon X^2 \setminus \Delta \to \C$ for which we have for some $\alpha > 0$ and $C < \infty$ that 
\begin{displaymath}
|K(x,y)| \le C\min\Big(\frac{1}{\lambda(x, \rho(x,y))}, \frac{1}{\lambda(y, \rho(x,y))}\Big), \qquad x \ne y,
\end{displaymath}
\begin{displaymath}
|K(x,y) - K(x',y)| \le C\frac{\rho(x,x')^{\alpha}}{\rho(x,y)^{\alpha}\lambda(x, \rho(x,y))}, \qquad \rho(x,y) \ge C\rho(x,x'),
\end{displaymath}
and
\begin{displaymath}
|K(x,y) - K(x,y')| \le C\frac{\rho(y,y')^{\alpha}}{\rho(x,y)^{\alpha}\lambda(y, \rho(x,y))}, \qquad \rho(x,y) \ge C\rho(y,y').
\end{displaymath}
Let $T\colon f \mapsto Tf$ be a linear operator acting on some functions $f$ (which we shall specify in more detail later). It is called a Calder\'on--Zygmund operator with kernel $K$ if
\begin{displaymath}
Tf(x) = \int_X K(x,y)f(y)\,d\mu(y)
\end{displaymath}
for $x$ outside the support of $f$.

\subsection{Accretivity}
A function $b \in L^{\infty}(\mu)$ is called accretive if Re$\,b \ge a > 0$ almost everywhere. We can also make do with the following weaker form of accretivity:
$|\int_A b\,d\mu| \ge a\mu(A)$ for all Borel sets $A$ which satisfy the condition that $B \subset A \subset CB$ for some ball $B = B(A)$, where $C$ is some large
constant which depends on the quasimetric $\rho$. (One can e.g. take $C=500$ if dealing with metrics).

The point is to have the above estimate whenever $A$ is one of the ``cubes'' to be constructed below, but there is no easy explicit description of what kind of sets they actually are.

\subsection{Weak boundedness property}
An operator $T$ is said to satisfy the weak boundedness property if $|\langle T\chi_B, \chi_B\rangle| \le A\mu(\Lambda B)$ for all balls $B$ and for some fixed constants $A > 0$ and $\Lambda > 1$.
Here $\langle\cdot\,,\, \cdot \rangle$ is the bilinear duality $\langle f, g\rangle = \int fg\,d\mu$. Let us denote the smallest admissible constant above by $\|T\|_{WBP_{\Lambda}}$.

In the $Tb$ theorem, the weak boundedness property is demanded from the operator $M_{b_2}TM_{b_1}$, where
$b_1$ and $b_2$ are accretive functions and $M_b \colon f \mapsto bf$.

\subsection{BMO and RBMO}
We say that $f \in L^1_{\textrm{loc}}(\mu)$ belongs to BMO$^p_{\kappa}(\mu)$, if for any ball $B \subset X$ there
exists a constant $f_B$ such that
\begin{displaymath}
\Big( \int_B |f-f_B|^p\,d\mu\Big)^{1/p} \le L\mu(\kappa B)^{1/p},
\end{displaymath}
where the constant $L$ does not depend on $B$.

Let $\varrho > 1$. A function $f \in L^1_{\textrm{loc}}(\mu)$ belongs to RBMO$(\mu)$ if there exists a constant $L$, and for every ball $B$, a constant $f_B$, such that one has
\begin{displaymath}
\int_B |f-f_B| \,d\mu \le L\mu(\varrho B),
\end{displaymath}
and, whenever $B \subset B_1$ are two balls,
\begin{displaymath}
|f_B - f_{B_1}| \le L\Big(1 + \int_{2B_1 \setminus B} \frac{1}{\lambda(c_B, d(x,c_B))}\,d\mu(x)\Big).
\end{displaymath}
We do not demand that $f_B$ be the average $\langle f \rangle_B = \frac{1}{\mu(B)}\int_B f\,d\mu$, and this is actually important in the RBMO$(\mu)$-condition. The useful thing here is that the space RBMO$(\mu)$ is independent of the choice of parameter $\varrho > 1$ and satisfies the John--Nirenberg inequality. For these results in our setting, see \cite{Hy}.
The norms in these spaces are defined in the obvious way as the best constant $L$.

We do not really need the RBMO$(\mu)$ space here as we formulate our main theorem with respect to BMO$^2_{\kappa}(\mu)$ rather than BMO$^1_{\kappa}(\mu)$. However, some reductions are possible here, and we comment on this after the formulation of our main theorem.

\subsection{Vinogradov notation and implicit constants}
The notation $f \lesssim g$ is used synonymously with $f \le Gg$ for some constant $G$. We also use $f \sim g$ if $f \lesssim g \lesssim f$. We now specify on the parameters on which the  implied constant  $G$ is allowed to depend on in this notation.
We let it depend on $N$ and $A_0$, which are related to the space $X$, and on $C_{\lambda}$, which
is related to the measure $\mu$. Next, we let $D$ depend on $C$ and $\alpha$, the constants from the kernel estimates, and on the constants $A$ and $\Lambda$ related to the weak boundedness property. Then we let $G$
depend on $a$, the constant related to the accretivity assumption of the test functions involved, and on the $L^{\infty}$-norms of the same test functions. We also let $G$ depend on the BMO parameters. It is an inconvenient fact
that one needs so many constants. Several more auxiliary parameters will be build of the aforementioned ones. Also, there will be quite a few other parameters that are not swallowed by this notation.

\medskip

We now formulate our main theorem in full detail.

\begin{thm}\label{maintheorem}
Let $(X,\rho)$ be a geometrically doubling regular quasimetric space which is equipped with an upper doubling measure $\mu$.
Let $T$ be an $L^2(\mu)$-bounded Calder\'on--Zygmund operator with a standard kernel $K$,
let $b_1$ and $b_2$ be two accretive functions, and $\kappa,\Lambda > 1$. Then
\begin{equation*}
  \|T\|\lesssim \|Tb_1\|_{BMO^2_{\kappa}(\mu)}+\|T^*b_2\|_{BMO^2_{\kappa}(\mu)}
    +\|M_{b_2}TM_{b_1}\|_{WBP_{\Lambda}}+1,
\end{equation*}
where the first three terms on the right are in turn dominated by $\|T\|$.

Here all the estimates depend on the quasimetric space $(X,\rho)$ through the constants $N$, $A_0$ and the function $\epsilon \mapsto A(\epsilon)$,
the measure $\mu$ through the constant $C_{\lambda}$, the test functions $b_1$ and $b_2$ through
the $L^{\infty}$-norms and the accretivity constants, and the kernel $K$ through the constants $C$ and $\alpha$ appearing in the standard estimates.
\end{thm}
\begin{rem}
\begin{itemize}
\item[(i)]
We have made the assumption that $T$ is a priori known to be a bounded operator on $L^2(\mu)$, but
one can reduce to this situation in several well-known ways (and sometimes no reduction is necessary). In what follows, we shall take the two BMO-norms and weak boundedness constant to be $1$, and  show that $\|T\| \lesssim 1$ with the conventions agreed upon above. It will be clear from the proof that the dependence on the mentioned quantities is of the asserted form.
The converse direction of the theorem is standard.
\item[(ii)]
At least in the case that $\rho$ is an honest metric, one can work with the larger space BMO$^1_{\kappa}(\mu)$ too; see \cite{NTV} and \cite{Hy}. Here one passes through the RBMO$(\mu)$ space, uses the John--Nirenberg inequality there, and then returns to the BMO$^2_{\kappa}(\mu)$ setting. In other words, $Tb_1 \in \textrm{BMO}^1_{\kappa}(\mu)$ implies that actually $Tb_1 \in \textrm{BMO}^2_{\kappa}(\mu)$. (This utilizes the fact that $Tb_1$ is in the range of the Calder\'on--Zygmund operator $T$ e.g.\ via the weak boundedness property; nothing of this sort holds for a general $f\in\textrm{BMO}_{\kappa}^1(\mu)$).
\item[(iii)]
As will be explained in the following section, all assumptions except possibly for the weak boundedness property are stable under the change to an equivalent quasimetric. Hence the only place were the regularity assumption on $\rho$ plays a role is in the balls involved in the weak boundedness property. We do not know whether the theorem remains true if the weak boundedness property is assumed with respect to balls of an irregular quasimetric. On the other hand, since these balls could be quite wild, it is questionable if such a testing condition would even be very useful.
\end{itemize}
\end{rem}

\section{Reduction to metric spaces}

A result of Mac\'ias and Segovia (see the proof of \cite[Theorem~2]{MS}) implies that there is a metric $d$ and a constant $\beta \ge 1$ such that
$2^{-\beta}d(x,y)^{\beta} \le \rho(x,y) \le 4^{\beta}d(x,y)^{\beta}$ for all $x,y \in X$. Here $3A_0^2 = 2^{\beta}$.
One checks that $(X,d)$ is then geometrically doubling with the constant $N16^{\beta n}$. Choosing $N$ large enough in the first place
lets us again use the same constant for both spaces. The measure $\mu$ is also upper doubling with respect to $d$ with the function
$(x,r) \mapsto \lambda(x, 4^{\beta}r^{\beta})$. As we no longer have any use for the original $\lambda$, we replace it with this one.

It also follows that $Tb_1$ and $T^*b_2$ belong to the space BMO$^{2}_{\kappa}(\mu)$ with respect to $d$-balls, if we just simply replace the original $\kappa$ by
$8\kappa^{1/\beta}$. Of course, if the accretivity is assumed in the form Re$\, b \ge a > 0$, this requires no modifications. If we assume it in the weaker form, one sees that we can e.g. take
$C = 4000^{\beta}$. Then we have that if $B_d \subset A \subset 500B_d$ for some $d$-ball $B_d$, then $B_{\rho} \subset A \subset CB_{\rho}$ for some $\rho$-ball. The proof that follows shows
that the constant $500$ works, as certain dyadic cubes with respect to $d$ are sets of this form.

There seems to be no easy way to immediately conclude the weak boundedness property also for the $d$-balls. Thus, we shall not even attempt anything of this sort. Instead, we shall explicitly use the quasimetric $\rho$ in a certain random ball covering, and as the reader will see, this circumvents the problem.

It remains to speak about the kernel estimates with respect to $d$ and the new $\lambda$ (which works for $d$-balls). An easy calculation using the facts that $\lambda$ is non-decreasing and doubling shows
that the first kernel estimate holds, just with a larger constant $C$. The rest of the kernel estimates also hold, just with the original $\alpha$ replaced by $\beta\alpha$ and demanding $C$ to be large enough.
This ends our reduction. As stated, for the most part we from now on deal with just the metric space $(X,d)$, and use the original $\rho$ only in one carefully indicated place.

\section{Dyadic systems of cubes}

We now provide a variant of Christ's \cite{Ch} construction of dyadic cubes in a metric space. His original result was formulated assuming the presence of a doubling measure, but most of the argument actually employs geometric doubling only. However, Christ only proved the covering property of his cubes in an a.e.~sense with respect to the doubling reference measure. We want to avoid this, which leads us to the construction of a system of ``half-open'' cubes, which exactly partition the whole space at every length scale, just like in $\R^n$.

The construction involves a parameter $\delta \le 1/1000$. For each $k \in \Z$ we are given a collection of points $x^k_{\alpha}$ for which
$d(x^k_{\alpha}, x^k_{\beta}) \ge \delta^k/8$ for all $\alpha \ne \beta$ and $\min_{\alpha} d(x, x^k_{\alpha}) < 4\delta^k$ for every $x \in X$. The parameters $1/8$ and $4$ are used simply because they will do in a certain randomization procedure of points (see Sec.~\ref{sec:randomDyadic}).

Let us construct a certain transitive relation $\le$ among the pairs $(k,\alpha)$, following \cite{Ch}. For each $(k, \alpha)$ there exists at least one $\beta$ for which $d(x^k_{\alpha},x^{k-1}_{\beta}) < 4\delta^{k-1}$. Also, there exists
at most one $\beta$ for which $d(x^k_{\alpha}, x^{k-1}_{\beta}) < \delta^{k-1}/16$. The ordering $\le$ is constructed using the rules we now describe. Consider any pair $(k, \alpha)$. Check first whether there exists $\beta$ so that
$d(x^k_{\alpha}, x^{k-1}_{\beta}) < \delta^{k-1}/16$. If so, set $(k, \alpha) \le (k-1,\beta)$ and $(k, \alpha) \not \le (k-1, \gamma)$ for $\gamma \ne \beta$. Otherwise, choose any $\beta$ for which
$d(x^k_{\alpha},x^{k-1}_{\beta}) < 4\delta^{k-1}$, and set $(k, \alpha) \le (k-1,\beta)$ and $(k, \alpha) \not \le (k-1, \gamma)$ for $\gamma \ne \beta$. Extend by transitivity to obtain a partial ordering.

The dyadic cubes of Christ are defined by
\begin{displaymath}
Q^k_{\alpha} = \bigcup_{(\ell, \beta) \le (k, \alpha)} B(x^{\ell}_{\beta}, \delta^{\ell}/100).
\end{displaymath}
However, we aim to replace them by the ``half-open'' cubes advertised before.

One can easily check that $\overline{Q}^k_{\alpha} \subset B(x^k_{\alpha}, 5\delta^k)$. Also, we have the property that if $\ell \ge k$, then either $Q^{\ell}_{\beta} \subset Q^k_{\alpha}$ or $Q^{\ell}_{\beta} \cap \overline{Q}^k_{\alpha} = \emptyset$ --
these still follow more or less as in \cite{Ch}. We now state and prove a number of lemmata relevant to our modification.

\begin{lem}\label{fullcoverbyclosures}
For all $k \in \Z$, the closures $\overline{Q}^k_{\alpha}$ cover $X$.
\end{lem}

\begin{proof}
Let $x \in X$ be arbitrary. For each $m \ge k$ we find some $x^m_{\beta} =: x_m$ so that $d(x_m, x) < 4\delta^m$. This especially implies that $x_m \to x$.
We know that $x_m \in Q^k_{\alpha(m)}$ for some $\alpha(m)$. We have $d(x^k_{\alpha(m)}, x) \le d(x^k_{\alpha(m)}, x_m) + d(x_m, x) <  9\delta^k$.
As $X$ is geometrically doubling, this implies that $\alpha(m)$ can take only finitely many values. In particular, one finds
an infinite subsequence with $x_m \in Q^k_{\alpha}$ for some fixed $\alpha$, and so $x \in \overline{Q}^k_{\alpha}$.
\end{proof}
We then note that as the collection $(\overline{Q}^k_{\alpha})_{\alpha}$ is locally finite, any union of them is closed. Let us then define the open dyadic cubes
\begin{displaymath}
\tilde Q^k_{\alpha} = X \setminus \bigcup_{\beta \ne \alpha} \overline{Q}^k_{\beta},
\end{displaymath}
and note that, by what we have already seen, there holds $Q^k_{\alpha} \subset \tilde Q^k_{\alpha} \subset \overline{Q}^k_{\alpha}$.
\begin{lem}
If $\ell \ge k$, we have
\begin{displaymath}
\overline{Q}^k_{\alpha} = \bigcup_{\sigma:\,(\ell, \sigma) \le (k, \alpha)} \overline{Q}^{\ell}_{\sigma}.
\end{displaymath}
\end{lem}
\begin{proof}
As this is obvious for $\ell = k$, we take $\ell > k$. Write
\begin{displaymath}
Q^k_{\alpha} = \bigcup_{\sigma:\,(\ell, \sigma) \le (k, \alpha)} Q^{\ell}_{\sigma} \cup \mathop{\bigcup_{(m, \sigma) \le (k, \alpha)}}_{m < \ell} B(x^m_{\sigma}, \delta^m/100).
\end{displaymath}
As the union on the right-hand side is finite, we have
\begin{displaymath}
\overline{Q}^k_{\alpha} = \bigcup_{\sigma:\,(\ell, \sigma) \le (k, \alpha)} \overline{Q}^{\ell}_{\sigma} \cup \mathop{\bigcup_{(m, \sigma) \le (k, \alpha)}}_{m < \ell} \overline{B}(x^m_{\sigma}, \delta^m/100).
\end{displaymath}
Thus, it suffices to prove that
\begin{displaymath}
\mathop{\bigcup_{(m, \sigma) \le (k, \alpha)}}_{m < \ell} \overline{B}(x^m_{\sigma}, \delta^m/100) \subset \bigcup_{\sigma:\,(\ell, \sigma) \le (k, \alpha)} \overline{Q}^{\ell}_{\sigma}.
\end{displaymath}
Fix some $(m, \sigma) \le (k, \alpha)$, where $m < \ell$, and consider a point $x \in \overline{B}(x^m_{\sigma}, \delta^m/100)$. Since the sets $(\overline{Q}^{\ell}_{\beta})_{\beta}$ cover the whole $X$ by Lemma \ref{fullcoverbyclosures},
there is some $\beta$ for which $x \in \overline{Q}^{\ell}_{\beta}$.

It remains to show that $(\ell, \beta) \le (k, \alpha)$. Fix the $\gamma$ for which $(\ell, \beta) \le (m+1, \gamma)$. We have
\begin{align*}
d(x^m_{\sigma}, x^{m+1}_{\gamma}) &\le d(x^m_{\sigma}, x) + d(x, x^{\ell}_{\beta}) + d(x^{\ell}_{\beta}, x^{m+1}_{\gamma}) \\
&\le \delta^m/100 + 5\delta^{\ell} + 5\delta^{m+1} < \delta^m/50 < \delta^m/16
\end{align*}
proving that $(m+1, \gamma) \le (m, \sigma)$. Since $(\ell, \beta) \le (m+1, \gamma)$ and $(m, \sigma) \le (k,\alpha)$, it holds $(\ell, \beta) \le (k,\alpha)$.
\end{proof}
\begin{lem}
If $\ell \ge k$, then either $\tilde Q^{\ell}_{\beta} \subset \tilde Q^k_{\alpha}$, or else $\tilde Q^{\ell}_{\beta} \cap \overline{Q}^k_{\alpha} = \emptyset$.
\end{lem}
\begin{proof}
Suppose that $\tilde Q^{\ell}_{\beta} \cap \overline{Q}^k_{\alpha} \ne \emptyset$. The previous lemma gives that $\tilde Q^{\ell}_{\beta} \cap \overline{Q}^{\ell}_{\sigma} \ne \emptyset$ for some
$(\ell, \sigma) \le (k, \alpha)$. By the definition of open cubes, we have to have $\beta = \sigma$, and thus $(\ell, \beta) \le (k, \alpha)$. Since there is a unique $\alpha$ with this property, it has to be that
$\tilde Q^{\ell}_{\beta} \cap \overline{Q}^k_{\gamma} = \emptyset$ for all $\gamma \ne \alpha$. This implies that
\begin{displaymath}
\tilde Q^{\ell}_{\beta} \subset X \setminus \bigcup_{\gamma \ne \alpha} \overline{Q}^k_{\gamma} = \tilde Q^k_{\alpha}.
\end{displaymath}
\end{proof}
We are now ready to construct the exact partition of $X$ using ``half-open'' cubes.
\begin{thm}
There exist sets $\hat Q^k_{\alpha}$, obtained from closed and open sets by finitely many operations, such that $\tilde Q^k_{\alpha} \subset \hat Q^k_{\alpha} \subset \overline{Q}^k_{\alpha}$,
\begin{displaymath}
X = \bigcup_{\alpha} \hat Q^k_{\alpha}
\end{displaymath}
for every $k \in \Z$ and if $\ell \ge k$ then either $\hat Q^k_{\alpha} \cap \hat Q^{\ell}_{\beta} = \emptyset$ or $\hat Q^{\ell}_{\beta} \subset \hat Q^k_{\alpha}$. Moreover, for every $\ell \ge k$ we have
\begin{displaymath}
\hat Q^k_{\alpha} = \bigcup_{\beta:\,(\ell, \beta) \le (k, \alpha)} \hat Q^{\ell}_{\beta}.
\end{displaymath}
\end{thm}
\begin{proof}
We may assume that $\alpha \in \N$ (just enumerate them for each $k$). For $k=0$, define
\begin{displaymath}
\hat Q^0_0 = \overline{Q}^{0}_0, \,\,\, \hat Q^0_{\alpha} = \overline{Q}^0_{\alpha} \setminus \bigcup_{\beta=0}^{\alpha-1} \hat Q^0_{\beta}, \,\,\, \alpha \ge 1.
\end{displaymath}
For $k < 0$ define
\begin{displaymath}
\hat Q^k_{\alpha} = \bigcup_{\beta:\,(0, \beta) \le (k, \alpha)} \hat Q^0_{\beta}.
\end{displaymath}
Finally, for $k > 0$ we proceed by induction as follows. Suppose that the cubes $\hat Q^{\ell}_{\alpha}$, $\ell \le k -1$, are already defined. For every $\alpha$, consider the finitely many pairs
$(k, \beta) \le (k-1, \alpha)$, temporarily relabel them $\beta = 0, 1, \ldots$ (up to some finite number), and set
\begin{displaymath}
\hat Q^k_0 = \hat Q^{k-1}_{\alpha} \cap \overline{Q}^k_0, \,\,\, \hat Q^k_{\beta} = \hat Q^{k-1}_{\alpha} \cap \overline{Q}^k_{\beta} \setminus \bigcup_{\gamma = 0}^{\beta-1} \hat Q^k_{\gamma}, \,\,\, \beta \ge 1.
\end{displaymath}
All the properties follow.
\end{proof}
Note that it trivially holds that $B(x^k_{\alpha}, \delta^k/100) \subset \tilde Q^k_{\alpha}$. Our final lemma concerning solely these cubes will be of use later in the randomization procedure studied in detail in Sec.~\ref{sec:randomDyadic}.
\begin{lem}\label{rlemma}
Let $m \in \N$ and $\epsilon > 0$ be such that $500\epsilon \le \delta^m$. Suppose that $x \in \overline{Q}^k_{\alpha}$ is such that $d(x, X \setminus \tilde Q^k_{\alpha}) < \epsilon\delta^k$.
Then for any chain
\begin{displaymath}
(k+m, \sigma) = (k+m, \sigma_{k+m}) \le \ldots \le (k+1, \sigma_{k+1}) \le (k, \sigma_k)
\end{displaymath}
such that $x \in \overline{Q}^{k+m}_{\sigma_{k+m}}$, there holds $d(x^j_{\sigma_j}, x^i_{\sigma_i}) \ge \delta^j/500$ for all $k \le j < i \le k+m$.
\end{lem}
\begin{proof}
Let us denote, for brevity, $x^j_{\sigma_j}  = x_j$, $k \le j \le k+m$. Assume that $d(x_j,x_i) < \delta^j/500$ for some $k \le j < i \le k+m$. Let us first consider the case $\sigma_k = \alpha$.
We have $B(x_j, \delta^j/100) \subset \tilde Q^j_{\sigma_j} \subset \tilde Q^k_{\alpha}$, and thus
\begin{align*}
\delta^j/100 &\le d(x^j, X \setminus \tilde Q^k_{\alpha})\\
& \le d(x, X \setminus \tilde Q^k_{\alpha}) + d(x, x_i) + d(x_i, x_j) \\
& \le \epsilon\delta^k +5\delta^i + \delta^j/500 \\
& \le \delta^j/500 + \delta^j/200 + \delta^j/500 = 9\delta^j/1000 < \delta^j/100,
\end{align*} 
which is a contradiction. Suppose that $\sigma_k \ne \alpha$. Then $x \in \overline{Q}^{k+m}_{\sigma_{k+m}} \subset \overline{Q}^k_{\sigma_k}$. On the other hand, we have $x \in \overline{Q}^k_{\alpha} \subset X \setminus \tilde Q^k_{\sigma_k}$.
This implies that $d(x, X \setminus \tilde Q^k_{\sigma_k}) = 0 < \epsilon\delta^k$, so we are in the identical situation with $\alpha$ replaced by $\sigma_k$, and the same conclusion applies.
\end{proof}

We make the following important remark.
In all that follows, all the cubes will be ``half-open'', but the hat notation is no longer applied.

\section{Carleson's Embedding Theorem and Martingale Difference Decompositions}

We now state the Carleson embedding theorem and a related lemma in our setting. We omit the short proof of the following lemma; it is a straightforward adaptation of a known argument.

\begin{lem}
Let $a^k_{\alpha} \ge 0$ be non-negative numbers such that
\begin{displaymath}
\sum_{(\ell, \beta) \le (k,\alpha)} a^{\ell}_{\beta} \mu(Q^{\ell}_{\beta}) \le \mu(Q^k_{\alpha})
\end{displaymath}
for every $(k, \alpha)$. Suppose we also have some other collection of non-negative numbers $b^k_{\alpha} \ge 0$ and $b^*(x) = \sup_{Q^k_{\alpha} \ni x} b^k_{\alpha}$, when $x \in X$. Then it holds that
\begin{displaymath}
\sum_{(k, \alpha)} b^k_{\alpha}a^k_{\alpha}\mu(Q^k_{\alpha}) \le \int_X b^*\,d\mu.
\end{displaymath}
\end{lem}

\begin{thm}[Carleson's Embedding Theorem]
Suppose we are given non-negative numbers $a^k_{\alpha} \ge 0$ such that
\begin{displaymath}
\sum_{(\ell, \beta) \le (k,\alpha)} a^{\ell}_{\beta} \mu(Q^{\ell}_{\beta}) \le \mu(Q^k_{\alpha})
\end{displaymath}
for every $(k, \alpha)$. Then we have for every $f \in L^2(\mu)$ that
\begin{displaymath}
\sum_{(k, \alpha)} |\langle f\rangle_{Q^k_{\alpha}}|^2 a^k_{\alpha}\mu(Q^k_{\alpha}) \lesssim \|f\|_{L^2(\mu)}^2.
\end{displaymath}
\end{thm}

\begin{proof}
Follows from the previous lemma and the fact that the dyadic maximal operator $M_d$ related to this set of cubes $Q^k_{\alpha}$
is of strong type $(2,2)$.
\end{proof}

We continue to define the martingale difference decomposition of a function $f$, a tool fundamental to our study. Set
\begin{align*}
E_kf &= \sum_{\alpha} \langle f \rangle_{Q^k_{\alpha}}\chi_{Q^k_{\alpha}}, \\
E_{Q^k_{\alpha}}f &= \chi_{Q^k_{\alpha}}E_k f, \\
\Delta_kf &= E_{k+1}f - E_kf, \\
\Delta_{Q^k_{\alpha}}f & = \chi_{Q^k_{\alpha}}\Delta_k f.
\end{align*}
If $b$ is accretive, one has the $b$-adapted versions $E_k^b$, $E_{Q^k_{\alpha}}^b$, $\Delta_k^b$ and $\Delta_{Q^k_{\alpha}}^b$:
\begin{align*}
E_k^bf &= \sum_{\alpha} \langle f \rangle_{Q^k_{\alpha}}\langle b \rangle_{Q^k_{\alpha}}^{-1} \chi_{Q^k_{\alpha}}b, \\
E_{Q^k_{\alpha}}^bf &= \chi_{Q^k_{\alpha}}E_k^b f, \\
\Delta_k^bf &= E_{k+1}^bf - E_k^bf, \\
\Delta_{Q^k_{\alpha}}^bf & = \chi_{Q^k_{\alpha}}\Delta_k^b f.
\end{align*}
We have for any $m$ the decompositions
\begin{align*}
f &= \sum_{(k, \alpha): \, k \ge m} \Delta_{Q^k_{\alpha}}f + \sum_{\alpha} E_{Q^m_{\alpha}}f \\
 & = \sum_{(k, \alpha): \, k \ge m} \Delta_{Q^k_{\alpha}}^bf + \sum_{\alpha} E_{Q^m_{\alpha}}^bf,
\end{align*}
and it also holds that
\begin{align*}
\|f\|_{L^2(\mu)}^2 &= \sum_{(k, \alpha): \, k \ge m} \|\Delta_{Q^k_{\alpha}}f\|_{L^2(\mu)}^2 + \sum_{\alpha} \|E_{Q^m_{\alpha}}f\|_{L^2(\mu)}^2 \\
&\sim \sum_{(k, \alpha): \, k \ge m} \|\Delta_{Q^k_{\alpha}}^bf\|_{L^2(\mu)}^2 + \sum_{\alpha} \|E_{Q^m_{\alpha}}^bf\|_{L^2(\mu)}^2.
\end{align*}
Here the implied constants depends only on $b$ via its accretivity constant and $L^{\infty}$-norm. This last fact follows from some algebraic manipulations, Carleson's embedding theorem and duality as in \cite[Lemma 4.1]{NTV}.

In what follows we are given two dyadic systems $\mathcal{D}$ and $\mathcal{D}'$ of ``half-open'' cubes as constructed above,
and we associate the accretive function $b_1$ to $\mathcal{D}$ and the accretive function $b_2$ to $\mathcal{D}'$.
We denote the $\mathcal{D}$-cubes by $Q^k_{\alpha}$ and the $\mathcal{D}'$-cubes by $R^m_{\gamma}$. Since for many purposes this is too heavy a notation, we agree that
$\ell(Q^k_{\alpha}) = \delta^k$ and $\textrm{gen}(Q^k_{\alpha}) = k$ (and similarly for the $\mathcal{D}'$-cubes), and more often than not denote simply $Q = Q^k_{\alpha}$ and $R = R^m_{\gamma}$.
Also, set $C_0 = 10$, $C_1 = 1/100$ and $C_3 = 4$, so that $d(Q^k_{\alpha}) < C_0\delta^k$, $B(x^k_{\alpha}, C_1\delta^k) \subset Q^k_{\alpha}$ and $\min_{\alpha} d(x, x^k_{\alpha}) < C_3\delta^k$ (and similarly 
for the other grid $\mathcal{D}'$). Note also that all these cubes are sets like in the definition of accretivity.

We have now disposed of the preliminaries, and will begin the task of estimating the operator $T$. As the reader probably already knows, the idea is to write the adapted martingale difference decompositions for two functions $f$ and $g$ with respect to the grids $\mathcal{D}$ and $\mathcal{D}'$ respectively, decompose $\langle Tf, g\rangle$, and study the various pairings $\langle T\Delta^{b_1}_Q f, \Delta^{b_2}_R g\rangle$ thus introduced.
Note that the theorems and lemmas formulated below do not cover all the cases per se (we mostly consider $\ell(Q)\leq\ell(R)$ only), but combined with symmetry they do.
Everything will be brought together to prove the $Tb$ theorem in the very end. We follow the outline given by the most general aspects of \cite{NTV}, and the main contributions are in the details.

\section{Separated cubes}

Here we deal with well-separated cubes $Q \in \mathcal{D}$ and $R \in \mathcal{D}'$.

\begin{lem}  
Let $Q \in \mathcal{D}$, $R \in \mathcal{D}'$, $\ell(Q) \le \ell(R)$ and $d(Q,R) \ge CC_0\ell(Q)$. Let $\varphi_Q$ and $\psi_R$ be $L^2(\mu)$ functions supported by the cubes $Q$ and $R$ respectively
and assume $\int \varphi_Q = 0$. We have the estimate
\begin{displaymath}
|\langle T\varphi_Q, \psi_R\rangle| \lesssim \frac{\ell(Q)^{\alpha}}{d(Q,R)^{\alpha}\sup_{z \in Q} \lambda(z, d(Q,R))}\mu(Q)^{1/2}\mu(R)^{1/2}\|\varphi_Q\|_{L^2(\mu)}\|\psi_R\|_{L^2(\mu)}.
\end{displaymath}
\end{lem}

\begin{proof}
This follows from the second kernel estimate via the facts that $\varphi_Q$ has zero integral and $d(Q,R) \ge CC_0\ell(Q)$.
\end{proof}

A reader familiar with the original proof may recall that a condition of the type $d(Q, R) \ge \ell(Q)^{\gamma}\ell(R)^{1-\gamma}$ plays a key role. The correct choice for $\gamma$ in our situation is has the same algebraic expression as in \cite{NTV},
\begin{displaymath}
\gamma := \frac{\alpha}{2(\alpha + d)},
\end{displaymath}
where we recall that $d := \log_2 C_{\lambda}$ in our setting. We then have the familiar relation $\gamma d + \gamma \alpha = \alpha/2$. Also, set
$D(Q,R) = \ell(Q) + \ell(R) + d(Q,R)$.

\begin{lem}\label{ch6l}
Assume that $Q \in \mathcal{D}$, $R \in \mathcal{D}'$, $\ell(Q) \le \ell(R)$, $d(Q,R) \ge CC_0\ell(Q)$ and, in addition, $d(Q,R) \ge \ell(Q)^{\gamma}\ell(R)^{1-\gamma}$.
Let $\varphi_Q$ and $\psi_R$ be $L^2(\mu)$ functions supported by the cubes $Q$ and $R$ respectively and assume $\int \varphi_Q = 0$. We have the estimate
\begin{displaymath}
|\langle T\varphi_Q, \psi_R\rangle| \lesssim \frac{\ell(Q)^{\alpha/2}\ell(R)^{\alpha/2}}{D(Q,R)^{\alpha}\sup_{z \in Q} \lambda(z, D(Q,R))}\mu(Q)^{1/2}\mu(R)^{1/2}\|\varphi_Q\|_{L^2(\mu)}\|\psi_R\|_{L^2(\mu)}.
\end{displaymath}
\end{lem}

\begin{proof}
Consider first the case $d(Q,R) \ge \ell(R)$. Then $d(Q,R) \ge D(Q,R)/3$ and so $\lambda(z, d(Q,R)) \ge \lambda(z, D(Q,R)/3) \gtrsim \lambda(z, D(Q,R))$ as $\lambda$ is doubling. The estimate then follows from the previous lemma.

Let us now assume that $d(Q,R) \le \ell(R)$. Note that $C_{\lambda} = 2^d$ so that
\begin{displaymath}
C_{\lambda}^{-\gamma \log_2 \frac{\ell(R)}{\ell(Q)}} = \Big(\frac{\ell(R)}{\ell(Q)}\Big)^{-\gamma d}.
\end{displaymath}
We have
\begin{align*}
  \lambda(z, \ell(R))
  &= \lambda(z, (\ell(R)/\ell(Q))^{\gamma} \ell(Q)^{\gamma} \ell(R)^{1-\gamma}) \\
  &\le C_{\lambda}^{\gamma \log_2 \frac{\ell(R)}{\ell(Q)} + 1} \lambda(z, \ell(Q)^{\gamma} \ell(R)^{1-\gamma})
\end{align*}
and thus
\begin{align*}
  \lambda(z, d(Q,R))
  &\ge \lambda(z, \ell(Q)^{\gamma}\ell(R)^{1-\gamma}) \\
  &\gtrsim C_{\lambda}^{-\gamma \log_2 \frac{\ell(R)}{\ell(Q)}} \lambda(z, \ell(R))
  = \Big(\frac{\ell(R)}{\ell(Q)}\Big)^{-\gamma d}\lambda(z, \ell(R)).
\end{align*}
This implies that
\begin{eqnarray*}
\frac{\ell(Q)^{\alpha}}{d(Q,R)^{\alpha}\lambda(z, d(Q,R))} & \lesssim & \frac{\ell(Q)^{\alpha}}{\ell(Q)^{\gamma \alpha}\ell(R)^{(1-\gamma)\alpha}\Big(\frac{\ell(R)}{\ell(Q)}\Big)^{-\gamma d}\lambda(z, \ell(R))} \\
& = & \frac{\ell(Q)^{\alpha - (\gamma d + \gamma \alpha)} \ell(R)^{\gamma d + \gamma \alpha}}{\ell(R)^{\alpha}\lambda(z, \ell(R))} \\
& = & \frac{\ell(Q)^{\alpha/2} \ell(R)^{\alpha/2}}{\ell(R)^{\alpha}\lambda(z, \ell(R))}.
\end{eqnarray*}
Furthermore, one now has $\ell(R) \ge D(Q,R)/3$ so that the claim follows from the previous lemma.
\end{proof}
We may now forget for the moment under which assumptions these estimates were achieved, and just study the matrix that we got.
Namely, let us define the matrix
\begin{displaymath}
T_{QR} = \frac{\ell(Q)^{\alpha/2}\ell(R)^{\alpha/2}}{D(Q,R)^{\alpha}\sup_{z \in Q} \lambda(z, D(Q,R))}\mu(Q)^{1/2}\mu(R)^{1/2},
\end{displaymath}
if $Q \in \mathcal{D}$, $R \in \mathcal{D}'$ and $\ell(Q) \le \ell(R)$, and
\begin{displaymath}
T_{QR} = 0
\end{displaymath}
otherwise.

\begin{prop}\label{ch6p}
Suppose we are given nonnegative constants $x_Q$ and $y_R$ for each $Q \in \mathcal{D}$ and $R \in \mathcal{D}'$. It holds
\begin{displaymath}
\sum_{Q, \, R} T_{QR}x_Qy_R \lesssim \Big( \sum_{Q} x_Q^2 \Big)^{1/2}\Big(\sum_{R}  y_R^2\Big)^{1/2}.
\end{displaymath}
\end{prop}

\begin{proof}
We assume first that $\ell(Q) = \delta^m \ell(R)$ for some $m = 0, 1, 2, \ldots$ and
then also that $\ell(R) = \delta^k$ for some $k \in \Z$. Define the kernel
\begin{displaymath}
K_{m,k}(x,y) = \sum_{\ell(Q) = \delta^{k+m},\, \ell(R) = \delta^k} T_{QR}\mu(Q)^{-1/2}\mu(R)^{-1/2}\chi_Q(x)\chi_R(y)
\end{displaymath}
and set
\begin{displaymath}
h_1 = \sum_{\ell(Q) = \delta^{k+m}} \mu(Q)^{-1/2}x_Q\chi_Q, \qquad h_2 = \sum_{\ell(R) = \delta^k} \mu(R)^{-1/2}y_R\chi_R.
\end{displaymath}
Note that
\begin{displaymath}
\sum_{\ell(Q) = \delta^{k+m}, \, \ell(R) = \delta^k} T_{QR}x_Qy_R  = \int_X \int_X K_{m,k}(x,y)h_1(x)h_2(y)\,d\mu(x)\,d\mu(y)
\end{displaymath}
and
\begin{displaymath}
\Big(\sum_{\ell(Q) = \delta^{k+m}} x_Q^2\Big)^{1/2} = \|h_1\|_{L^2(\mu)} \qquad \textrm{and} \qquad \Big(\sum_{\ell(R) = \delta^k} y_R^2\Big)^{1/2} = \|h_2\|_{L^2(\mu)}.
\end{displaymath}
Writing out the definitions one has
\begin{displaymath}
K_{m,k}(x,y) = \sum_{\ell(Q) = \delta^{k+m},\, \ell(R) = \delta^k} \frac{\ell(Q)^{\alpha/2}\ell(R)^{\alpha/2}}{D(Q,R)^{\alpha}\sup_{z \in Q} \lambda(z, D(Q,R))}\chi_Q(x)\chi_R(y).
\end{displaymath}

Consider some pair $(x,y)$. Then there exists one and only one pair $(Q_x,R_y)$ for which $\ell(Q_x) = \delta^{k+m}$,  $\ell(R_y) = \delta^k$, $x \in Q_x$ and $y \in R_y$, and so
\begin{displaymath}
K_{m,k}(x,y) =  \frac{\delta^{\alpha m/2}\delta^{\alpha k}}{D(Q_x,R_y)^{\alpha}\sup_{z \in Q_x} \lambda(z, D(Q_x,R_y))}.
\end{displaymath}
We want to prove that
\begin{equation}\label{akernele}
\int_X K_{m,k}(x,y)\,d\mu(x) \lesssim \delta^{\alpha m/2} \qquad \textrm{and} \qquad \int_X K_{m,k}(x,y)\,d\mu(y) \lesssim \delta^{\alpha m/2}.
\end{equation}

Let us only deal with the first term in detail---it is actually the bit harder of the two. The second integral is estimated basically in the same way as we now deal with the first integral,
but one does not have to go through the trouble of fiddling with the centers (just use directly that $\sup_{z \in Q_x} \lambda(z, D(Q_x,R_y)) \ge \lambda(x, D(Q_x,R_y))$).
We have
\begin{displaymath}
\int_X K_{m,k}(x,y)\,d\mu(x) = \Big(\int_{B(y, \delta^k)} + \int_{X \setminus B(y, \delta^k)}\Big) K_{m,k}(x,y)\,d\mu(x).
\end{displaymath}
Note that $D(Q_x, R_y) \ge \ell(R_y) = \delta^k$, and so
\begin{displaymath}
\int_{B(y, \delta^k)} K_{m,k}(x,y)\,d\mu(x) \le \delta^{\alpha m/2}\delta^{\alpha k} \int_{B(y, \delta^k)} \frac{d\mu(x)}{\delta^{\alpha k}\sup_{z \in Q_x} \lambda(z, \delta^k)}.
\end{displaymath}
We have that $Q_x \subset B(y, 2C_0\delta^k)$ for every $x \in B(y, \delta^k)$, and so $\sup_{z \in Q_x} \lambda(z, \delta^k) \ge \inf_{z \in B(y, 2C_0\delta^k)} \lambda(z, \delta^k)$. This yields
\begin{displaymath}
\int_{B(y, \delta^k)} K_{m,k}(x,y)\,d\mu(x) \le \delta^{\alpha m/2} \frac{\mu(B(y,\delta^k))}{\inf_{z \in B(y, 2C_0\delta^k)} \lambda(z, \delta^k)}.
\end{displaymath}
We then have that $B(y, \delta^k) \subset B(z, 3C_0\delta^k)$ for every $z \in B(y, 2C_0\delta^k)$ yielding that $\mu(B(y, \delta^k)) \le \mu(B(z, 3C_0\delta^k)) \le
\lambda(z, 3C_0\delta^k) \lesssim \lambda(z, \delta^k)$. This gives that
\begin{displaymath}
\int_{B(y, \delta^k)} K_{m,k}(x,y)\,d\mu(x) \lesssim \delta^{\alpha m/2}.
\end{displaymath}

We cannot directly employ Lemma~\ref{baest} to deal with the integral over $X \setminus B(y, \delta^k)$. However, we can use its proof together with similar gimmicks as with the previous term.
Note that $d(x,y) \lesssim D(Q_x, R_y)$ to get that
\begin{align*}
\int_{X \setminus B(y,\delta^k)} K_{m,k}&(x,y)\,d\mu(x)\\
& \lesssim \delta^{\alpha m/2}\delta^{\alpha k} \sum_{j=0}^{\infty} \int_{2^j\delta^k \le d(x,y) < 2^{j+1}\delta^k} \frac{d(x,y)^{-\alpha}}{\sup_{z \in Q_x} \lambda(z, d(x,y))}\,d\mu(x).
\end{align*}
If $2^j\delta^k \le d(x,y) < 2^{j+1}\delta^k$, we have $Q_x \subset B(y, C_02^{j+2}\delta^k)$, and therefore it holds $\sup_{z \in Q_x} \lambda(z, d(x,y)) \ge \inf_{z \in B(y, C_02^{j+2}\delta^k)} \lambda(z, 2^j\delta^k)$.
This establishes that
\begin{displaymath}
\int_{X \setminus B(y,\delta^k)} K_{m,k}(x,y)\,d\mu(x) \lesssim \delta^{\alpha m/2} \sum_{j=0}^{\infty} 2^{-\alpha j}\frac{\mu(B(y, 2^{j+1}\delta^k))}{\inf_{z \in B(y, C_02^{j+2}\delta^k)} \lambda(z, 2^j\delta^k)}.
\end{displaymath}
If $z \in B(y, C_02^{j+2}\delta^k)$, we have $B(y, 2^{j+1}\delta^k) \subset B(z, C_02^{j+3}\delta^k)$, and thus it holds
$\mu(B(y, 2^{j+1}\delta^k)) \le \mu(B(z, C_02^{j+3}\delta^k)) \le \lambda(z, C_02^{j+3}\delta^k) \lesssim \lambda(z, 2^j\delta^k)$. This proves that
\begin{displaymath}
\int_{X \setminus B(y,\delta^k)} K_{m,k}(x,y)\,d\mu(x) \lesssim \delta^{\alpha m/2}.
\end{displaymath}

We have established \eqref{akernele}. Schur's lemma then gives that
\begin{displaymath}
\int_X \int_X K_{m,k}(x,y)h_1(x)h_2(y)\,d\mu(x)\,d\mu(y) \lesssim \delta^{\alpha m/2} \|h_1\|_{L^2(\mu)}\|h_2\|_{L^2(\mu)}.
\end{displaymath}
As noted above this is the same as
\begin{displaymath}
\sum_{\ell(Q) = \delta^{k+m}, \, \ell(R) = \delta^k} T_{QR}x_Qy_R \lesssim \delta^{\alpha m/2} \Big(\sum_{\ell(Q) = \delta^{k+m}} x_Q^2\Big)^{1/2} \Big(\sum_{\ell(R) = \delta^k} y_R^2\Big)^{1/2}.
\end{displaymath}
Sum this over $k \in \Z$, use the Cauchy--Schwarz inequality, and then sum over $m = 0, 1, 2, \ldots$, to get that
\begin{displaymath}
\sum_{Q, \, R} T_{QR}x_Qy_R = \sum_{\ell(Q) \le \ell(R)} T_{QR}x_Qy_R \lesssim \Big( \sum_{Q} x_Q^2 \Big)^{1/2}\Big(\sum_{R}  y_R^2\Big)^{1/2}.
\end{displaymath}
\end{proof}

\section{Paraproducts and cubes well inside another cube}

We begin by proving the following lemma which is needed later in proving that a certain paraproduct is bounded.

\begin{lem}
Suppose that $Q \in \mathcal{D}$ is fixed and that $b$ is some pseudoaccretive function. It holds that
\begin{displaymath}
\mathop{\mathop{\sum_{R \in \mathcal{D}': \, R \subset Q}}_{\ell(R) \le \delta^r \ell(Q)}}_{d(R, X \setminus Q) \ge CC_0\kappa\ell(R)} \|\Delta^{b} _R (b\varphi)\|_{L^2(\mu)}^2 \lesssim \mu(Q)\|\varphi\|_{\textrm{BMO}^2_{\kappa}(\mu)}^2
\end{displaymath}
if $\varphi \in \textrm{BMO}^2_{\kappa}(\mu)$ and $r$ is so large that
\begin{displaymath}
\delta^r \le \frac{C_1}{CC_0\kappa + C_0 + C_3}.
\end{displaymath}
\end{lem}

Here the implied constants are exceptionally allowed to depend on the function $b$ in the obvious way.

\begin{proof}
We start by constructing a Whitney type decomposition using cubes and then we shall associate to each such cube in the covering a certain ball -- these balls, as we shall see,
will also have finite overlap (even when multiplied with the constant $\kappa$) because of the geometry of the construction and also because $X$ is geometrically doubling.

We first prove that there exist cubes $R \in \mathcal{D'}$ for which $R \subset Q$, $\ell(R) = \delta^r\ell(Q)$ and $d(R, X \setminus Q) \ge CC_0\kappa\ell(R)$.
Denote the center of $Q$ by $z$ and recall that we have $B(z, C_1\ell(Q)) \subset Q$. Choose some point $w$ which is the center of a $\mathcal{D}'$-cube $R$ of generation gen$(Q)+r$
and satisfies $d(w, z) < C_3\ell(R)$. Suppose that $x \in R \subset B(w, C_0\ell(R))$ and note that $d(x,z) \le d(x,w) + d(w, z) < (C_0 + C_3)\ell(R)$, that is
$R \subset B(z, (C_0+C_3)\ell(R))$. Suppose that $y \in X \setminus Q$ and $x \in R$. In this case we have $d(y,z) \ge C_1\ell(Q)$ and $d(x,z) < (C_0+C_3)\ell(R)$, which yields
\begin{displaymath}
d(x,y) \ge d(y,z) - d(x,z) \ge (C_1\delta^{-r} - C_0-C_3)\ell(R) \ge CC_0\kappa\ell(R).
\end{displaymath}
Thus, $d(R, X \setminus Q) \ge CC_0\kappa\ell(R)$. We choose all such cubes $R$. Then we choose all those $\mathcal{D}'$-cubes $R$ of the next generation which are not subcubes of the previously chosen cubes and which still
satisfy the condition that $d(R, X \setminus Q) \ge CC_0\kappa\ell(R)$. We continue in this way and obtain a disjoint collection of $\mathcal{D}'$-cubes $R$, which have the property that any cube $R'$ in the sum
\begin{displaymath}
\mathop{\mathop{\sum_{R' \in \mathcal{D}': \, R' \subset Q}}_{\ell(R') \le \delta^r \ell(Q)}}_{d(R', X \setminus Q) \ge CC_0\kappa\ell(R')} \|\Delta^{b} _{R'}(b\varphi)\|_{L^2(\mu)}
\end{displaymath}
is contained in one of them. To exploit the BMO condition we want a covering consisting of suitable balls though. To this end, we associate to each chosen cube $R$ the ball
$B_R = B(w_R, C_0\ell(R)) \supset R$ which is centered at the center $w_R$ of the cube $R$ and which has radius $C_0\ell(R)$.

We now wish to demonstrate that every $x \in Q$ belongs to $\lesssim 1$ balls $\kappa B_R$. We first prove that $x$ can belong to only $\lesssim 1$ balls $\kappa B_R$ associated to a fixed
generation $k \ge \textrm{gen}(Q) + r$ of the chosen cubes. Indeed, suppose that $x \in \kappa B_{R^k_i}$ for some collection $i$. This implies that $w_{R^k_i} \in B(x, C_0\kappa \delta^k)$ for all $i$.
This means that the ball $B(x, C_0\kappa \delta^k)$ contains the centers $w_{R^k_i}$ of the disjoint balls $B(w_{R^k_i}, C_1\delta^k) = B(w_{R^k_i}, C_1/(C_0\kappa) \cdot C_0\kappa \delta^k)$.
As $X$ is geometrically doubling, we have that $\#i \le N(C_1/(C_0\kappa))^{-n} \lesssim 1$.

We then prove that if $\kappa B_{R^k_i} \cap \kappa B_{R^l_j} \ne \emptyset$ (where $R^k_i$ and
$R^l_j$ are chosen cubes, $k, l > r$), then $|k - l| \lesssim 1$. This only utilizes the geometry of the construction. We prove a certain auxiliary estimate from which this follows.
Suppose that $R^k$ is a chosen cube of generation $k > \textrm{gen}(Q) + r$ and that
$x \in \kappa B_{R^k}$. As $R^k$ is a chosen cube, its dyadic parent $R^{k*}$ has to satisfy $d(R^{k*}, X \setminus Q) < CC_0\kappa \delta^{k-1}$. We take some point $y \in R^{k*}$ for which
$d(y, X \setminus Q) \le CC_0\kappa \delta^{k-1}$. Let $w_{R^k}$ be the center of $R^k$ and notice that we now have
\begin{eqnarray*}
d(x, X \setminus Q) & \le & d(x,y) + d(y, X \setminus Q) \\
& \le & d(x, w_{R^k}) + d(y, w_{R^k}) + d(y, X \setminus Q) \\
& \le & C_0\kappa \delta^k + C_0\delta^{k-1} + CC_0\kappa \delta^{k-1} \\
& \le & 3CC_0\kappa\delta^{k-1}. 
\end{eqnarray*}

To the other direction, there holds (we have $C \ge 2$)
\begin{displaymath}
d(x, X \setminus Q) \ge d(w_{R^k}, X \setminus Q) - d(x, w_{R^k}) \ge CC_0\kappa \delta^k  - C_0\kappa \delta^k \ge C_0\kappa \delta^k.
\end{displaymath}
We have established that if $R$ is a chosen cube for which gen$(R) = k > \textrm{gen}(Q) + r$, then
\begin{displaymath}
C_0\kappa \delta^k \le d(x, X \setminus Q) \le 3CC_0\kappa\delta^{k-1}
\end{displaymath}
holds for all $x \in \kappa B_R$. Taking logarithms one sees that this fixes $k$ to a certain finite range.

We have now done more than enough to show that for every $x \in Q$ one has
$\#\{R: x \in \kappa B_R\} \lesssim 1$, which we use to conclude that
\begin{displaymath}
\sum_R \mu(\kappa B_R) = \int_Q \sum_R \chi_{\kappa B_R} \,d\mu \lesssim \mu(Q),
\end{displaymath}
where we sum over the chosen cubes $R$. We have
\begin{eqnarray*}
\mathop{\mathop{\sum_{R' \in \mathcal{D}': \, R' \subset Q}}_{\ell(R') \le \delta^r \ell(Q)}}_{d(R', X \setminus Q) \ge CC_0\kappa\ell(R')} \|\Delta^{b} _{R'} (b\varphi)\|_{L^2(\mu)}^2
& = & \sum_R \sum_{R' \subset R} \|\Delta^{b}_{R'}(\chi_{B_R}b(\varphi - \varphi_{B_R}))\|_{L^2(\mu)}^2 \\
& \lesssim & \sum_R \int_{B_R} |\varphi - \varphi_{B_R}|^2\,d\mu \\
& \lesssim & \sum_R \mu(\kappa B_R)\|\varphi\|_{\textrm{BMO}^2_{\kappa}(\mu)}^2 \\
& \lesssim &\mu(Q)\|\varphi\|_{\textrm{BMO}^2_{\kappa}(\mu)}^2.
\end{eqnarray*}
\end{proof}

Let us define the paraproduct
\begin{displaymath}
\Pi f = \sum_{R' \in \mathcal{D}'} \mathop{\mathop{\sum_{Q' \in \mathcal{D}:\,Q' \subset R'}}_{\ell(Q') = \delta^r\ell(R')}}_{d(Q', X \setminus R') \ge CC_0\kappa \ell(Q')} (E_{R'} b_2)^{-1} \cdot E_{R'} f \cdot (\Delta_{Q'}^{b_1})^*(T^*b_2).
\end{displaymath}
We shall always assume that $r$ is at least as large as is required by the previous lemma. However, we make several further assumptions about it later. Of course, basically it could be fixed at the very beginning, and so it is not a problem
if we let the implied constants to depend also on $r$.
\begin{thm}\label{parabounded}
The paraproduct $\Pi$ is bounded on $L^2(\mu)$.
\end{thm}
\begin{proof}
This follows from the previous lemma via the $L^2(\mu)$ norm estimate related to the adapted martingale difference decomposition,
Carleson's embedding theorem and the fact that $(\Delta_{Q'}^{b_1})^*(T^*b_2) = b_1^{-1}\Delta_{Q'}^{b_1}(b_1T^*b_2)$.
\end{proof}
We introduce the concept of good cubes in more detail now. Recall the definition of $\gamma$ from the previous chapter.
Consider a cube $Q \in \mathcal{D}$. We say that $Q$ is good if for any cube $R \in \mathcal{D}'$ for which we have $\ell(Q) \le \delta^r\ell(R)$, we have
either $d(Q, R) \ge \ell(Q)^{\gamma}\ell(R)^{1-\gamma}$ or $d(Q, X \setminus R) \ge \ell(Q)^{\gamma}\ell(R)^{1-\gamma}$.
We denote this set by $\mathcal{D}_{\textrm{good}}$, and the rest are denoted by $\mathcal{D}_{\textrm{bad}}$.

Let us now fix $Q \in \mathcal{D}_{\textrm{good}}$ and $R \in \mathcal{D}'$ so that $Q \subset R$ and $\ell(Q) < \delta^r\ell(R)$. Let us also fix the child $R_1 \subset R$ for which $Q \subset R_1$
and let us denote the other children of $R$ by $R_i$. We have $\#i \lesssim 1$.
We assume that we are given two functions of the form
\begin{displaymath}
\varphi_Q = \mathop{\sum_{Q_j \in \mathcal{D}:\,Q_j \subset Q}}_{\ell(Q_j) = \delta\ell(Q)} A_{Q_j}\chi_{Q_j}b_1
\end{displaymath}
and
\begin{displaymath}
\psi_R = \mathop{\sum_{R_i \in \mathcal{D}':\,R_i \subset R}}_{\ell(R_i) = \delta \ell(R)} B_{R_i}\chi_{R_i}b_2.
\end{displaymath}
We demand that $\int \varphi_Q\,d\mu = 0$.
We aim to prove that
\begin{equation}\label{pares}
|\langle (T-\Pi^*)\varphi_Q, \psi_R\rangle| \lesssim \Big( \frac{\ell(Q)}{\ell(R)}\Big)^{\alpha/2}\Big(\frac{\mu(Q)}{\mu(R_1)}\Big)^{1/2}\|\varphi_Q\|_{L^2(\mu)}\|\psi_R\|_{L^2(\mu)}.
\end{equation}

In the previous chapter our matrix looked a little bit different from that in \cite{NTV}. However, \eqref{pares} is exactly of the same form as in \cite[Lemma 7.3]{NTV}.
We note that
\begin{displaymath}
\langle \varphi_Q, \Pi\psi_R\rangle = \sum_{R' \in \mathcal{D}'} \mathop{\mathop{\sum_{Q' \in \mathcal{D}:\,Q' \subset R'}}_{\ell(Q') = \delta^r\ell(R')}}_{d(Q', X \setminus R') \ge CC_0\kappa \ell(Q')} \langle \varphi_Q,
(E_{R'} b_2)^{-1} \cdot E_{R'} \psi_R \cdot (\Delta_{Q'}^{b_1})^*(T^*b_2) \rangle
\end{displaymath}
and that
\begin{displaymath}
\langle \varphi_Q, (E_{R'} b_2)^{-1} \cdot E_{R'} \psi_R \cdot (\Delta_{Q'}^{b_1})^*(T^*b_2) \rangle = \frac{\langle \psi_R \rangle_{R'}}{\langle b_2 \rangle_{R'}} \langle \Delta^{b_1}_{Q'} \varphi_Q, T^*b_2\rangle.
\end{displaymath}
Since $\varphi_Q$ is of this particular form and $\int \varphi_Q\,d\mu = 0$ one sees that $\Delta^{b_1}_{Q'} \varphi_Q = \varphi_Q$ if $Q'$ is the same cube as $Q$ and  $\Delta^{b_1}_{Q'} \varphi_Q = 0$ otherwise.
This implies in particular that in the non-trivial situation one must also have $R' \subset R_1$ and then $\frac{\langle \psi_R \rangle_{R'}}{\langle b_2 \rangle_{R'}} = B_{R_1}$. If a suitable $R' \supset Q$ exists (and there can only be one)
we thus have $\langle \varphi_Q, \Pi\psi_R\rangle = B_{R_1}\langle T\varphi_Q, b_2\rangle$. Let us now demonstrate that, indeed, such an $R'$ exists provided that we have chosen $r$ to be large enough.
As $Q \in \mathcal{D}_{\textrm{good}}$, $Q \subset R$ and $\ell(Q) < \delta^r\ell(R)$, we must have some $R' \in \mathcal{D}'$ so that $Q \subset R' \subset R_1 \subset R$,
$\ell(Q) = \delta^r\ell(R')$ and $d(Q, X \setminus R') \ge \ell(Q)^{\gamma}\ell(R')^{1-\gamma}$. In particular, demanding that $r$ is at least so large that $\delta^{-(1-\gamma)r} \ge CC_0\kappa$, we have
\begin{displaymath}
d(Q, X \setminus R') \ge \ell(Q)^{\gamma}\ell(R')^{1-\gamma} = \ell(Q)\delta^{-r(1-\gamma)} \ge CC_0\kappa\ell(Q).
\end{displaymath}
We have shown that $\langle \varphi_Q, \Pi\psi_R\rangle = B_{R_1}\langle T\varphi_Q, b_2\rangle$, and so
\begin{align*}
\langle (T-\Pi^*)\varphi_Q, \psi_R\rangle &= \langle T\varphi_Q, \psi_R - B_{R_1}b_2 \rangle \\ &= B_{R_1}\langle T\varphi_Q, (\chi_{R_1}-1)b_2\rangle  + \sum_{i \ne 1} B_{R_i}\langle T\varphi_Q, \chi_{R_i}b_2\rangle.
\end{align*}

We handle the first term first. Let us calculate (choosing some arbitrary point $z \in Q$)
{\setlength\arraycolsep{2pt}
\begin{eqnarray*}
|\langle T\varphi_Q, (\chi_{R_1}-1)b_2\rangle| & = & \Big| \int_{X \setminus R_1} \int_Q [K(x,y) - K(x,z)]\varphi_Q(y)b_2(x)\,d\mu(y)\,d\mu(x)\Big| \\
& \lesssim & \ell(Q)^{\alpha}\|\varphi_Q\|_{L^1(\mu)} \int_{X \setminus R_1} \frac{d(x,z)^{-\alpha}}{\lambda(z, d(x,z))}\,d\mu(x) \\
& \le & \ell(Q)^{\alpha}\|\varphi_Q\|_{L^1(\mu)} \int_{X \setminus B(z, d(Q, X \setminus R_1))} \frac{d(x,z)^{-\alpha}}{\lambda(z, d(x,z))}\,d\mu(x) \\
& \lesssim & \ell(Q)^{\alpha}\|\varphi_Q\|_{L^1(\mu)} d(Q, X \setminus R_1)^{-\alpha}.
\end{eqnarray*}
}
We used (for the kernel estimates) that $d(y,z) \le C_0\ell(Q)$ and $d(x,z) \ge d(Q, X \setminus R_1) \ge \ell(Q)^{\gamma}\ell(R_1)^{1-\gamma} \ge CC_0\ell(Q)$ as there is a gap of at least $r$ in the generations of $Q$ and $R_1$.
We then use that $d(Q, X \setminus R_1) \ge \ell(Q)^{\gamma}\ell(R_1)^{1-\gamma} \ge \ell(Q)^{1/2}\ell(R_1)^{1/2} \gtrsim \ell(Q)^{1/2}\ell(R)^{1/2}$ to get that
\begin{displaymath}
|\langle T\varphi_Q, (\chi_{R_1}-1)b_2\rangle| \lesssim \ell(Q)^{\alpha}\mu(Q)^{1/2}\|\varphi_Q\|_{L^2(\mu)}\ell(Q)^{-\alpha/2}\ell(R)^{-\alpha/2}.
\end{displaymath}
Note then that $|B_{R_1}| \lesssim \|\psi_R\|_{L^2(\mu)}\mu(R_1)^{-1/2}$ to infer
\begin{displaymath}
|B_{R_1}| |\langle T\varphi_Q, (\chi_{R_1}-1)b_2\rangle| \lesssim \Big( \frac{\ell(Q)}{\ell(R)}\Big)^{\alpha/2}\Big(\frac{\mu(Q)}{\mu(R_1)}\Big)^{1/2}\|\varphi_Q\|_{L^2(\mu)}\|\psi_R\|_{L^2(\mu)}.
\end{displaymath}
We then deal with the other $\#i \lesssim 1$ terms. This time we have, using estimates from chapter 6 (see the proof of Lemma \ref{ch6l}), for some fixed $z \in Q$ that
\begin{eqnarray*}
|B_{R_i}||\langle T\varphi_Q, \chi_{R_i}b_2\rangle| & = & |B_{R_i}|\Big| \int_{R_i} \int_Q [K(x,y) - K(x,z)]\varphi_Q(y)b_2(x)\,d\mu(y)\,d\mu(x)\Big| \\
& \lesssim & \|\varphi_Q\|_{L^2(\mu)}\|\psi_R\|_{L^2(\mu)}\mu(R_i)^{1/2}\mu(Q)^{1/2}\frac{\ell(Q)^{\alpha}}{d(Q,R_i)^{\alpha}\lambda(z,d(Q,R_i))} \\
& \lesssim & \|\varphi_Q\|_{L^2(\mu)}\|\psi_R\|_{L^2(\mu)}\mu(R_i)^{1/2}\mu(Q)^{1/2}\frac{\ell(Q)^{\alpha/2} \ell(R)^{-\alpha/2}}{\lambda(z, \ell(R_i))} \\
& = &  \Big( \frac{\ell(Q)}{\ell(R)}\Big)^{\alpha/2} \mu(Q)^{1/2} \frac{\mu(R_i)^{1/2}}{\lambda(z, \ell(R_i))}\|\varphi_Q\|_{L^2(\mu)}\|\psi_R\|_{L^2(\mu)}.
\end{eqnarray*}

We make the following deduction which uses the doubling property of $\lambda$:
\begin{displaymath}
\frac{\mu(R_i)^{1/2}}{\lambda(z, \ell(R_i))} \le \frac{\mu(R)^{1/2}}{\lambda(z, \ell(R_i))} \le  \frac{\lambda(z, C_0\ell(R))^{1/2}}{\lambda(z, \ell(R_i))} \lesssim \frac{1}{\lambda(z, C_0\ell(R))^{1/2}} \le \frac{1}{\mu(R_1)^{1/2}}.
\end{displaymath}
Insert this to the estimate from above to get that
\begin{displaymath}
|B_{R_i}||\langle T\varphi_Q, \chi_{R_i}b_2\rangle| \lesssim \Big( \frac{\ell(Q)}{\ell(R)}\Big)^{\alpha/2}\Big(\frac{\mu(Q)}{\mu(R_1)}\Big)^{1/2}\|\varphi_Q\|_{L^2(\mu)}\|\psi_R\|_{L^2(\mu)},
\end{displaymath}
and (\ref{pares}) follows. Set
\begin{displaymath}
T_{QR} = \Big( \frac{\ell(Q)}{\ell(R)}\Big)^{\alpha/2}\Big(\frac{\mu(Q)}{\mu(R_1)}\Big)^{1/2},
\end{displaymath}
if $Q \in \mathcal{D}_{\textrm{good}}$, $Q \subset R$ and $\ell(Q) < \delta^r\ell(R)$, $T_{QR} = 0$ otherwise.
\begin{prop}\label{ch7p}
The matrix $T_{QR}$ generates a bounded operator on $\ell^2$.
\end{prop}
\begin{proof}
This follows precisely as in \cite[Lemma 7.4]{NTV}, as one just has to deal with the measure $\mu$ not using any special assumptions about it.
\end{proof}

\section{Random almost-covering by balls}

We construct a probabilistic covering of a large portion of the space with balls, starting from some fixed size and going down in size but only for some controlled amount. This will be used as a substitute for a certain auxiliary third dyadic grid used in \cite[Sec. 10.2]{NTV} in connection with the weak boundedness property.
Here we need to explicitly work with our original quasimetric $\rho$, the reason being that the weak boundedness property does not transfer to the $d$-balls in any obvious way.

Let $0 < \vartheta < A_0^{-4}/32$. 
For each $k \in \Z$ fix some maximal collection $z^k_{\alpha} \in X$ for which $\rho(z^k_{\alpha}, z^k_{\beta}) \ge \vartheta^k$ for all $\alpha \ne \beta$.
This time we use the following transitive relation $\le$. For each $(k, \alpha)$ there exists at least one $\beta$ for which $\rho(z^k_{\alpha},z^{k-1}_{\beta}) < \vartheta^{k-1}$. Also, there exists
at most one $\beta$ for which $\rho(z^k_{\alpha}, z^{k-1}_{\beta}) < (2A_0)^{-1}\vartheta^{k-1}$.
The ordering $\le$ is constructed using the rules we now describe. Consider any pair $(k, \alpha)$. Check first whether there exists $\beta$ so that
$\rho(z^k_{\alpha}, z^{k-1}_{\beta}) < (2A_0)^{-1}\vartheta^{k-1}$. If so, set $(k, \alpha) \le (k-1,\beta)$ and $(k, \alpha) \not \le (k-1, \gamma)$ for $\gamma \ne \beta$. Otherwise, choose any $\beta$ for which
$\rho(z^k_{\alpha},z^{k-1}_{\beta}) < \vartheta^{k-1}$, and set $(k, \alpha) \le (k-1,\beta)$ and $(k, \alpha) \not \le (k-1, \gamma)$ for $\gamma \ne \beta$. Extend by transitivity.

We introduce another relation $\searrow$ now. Given any $(k, \alpha)$, pick one $\beta$ for which $(k+1,\beta) \le (k, \alpha)$  and set $(k, \alpha) \searrow (k+1, \beta)$ and
$(k, \alpha) \not \searrow (k+1, \gamma)$ for $\gamma \ne \beta$. This is the relation which we shall randomize in a natural way, and in this way we shall obtain in a random way a new collection of points for each level $k$.
Indeed, we shall essentially replace $z^k_{\alpha}$ with $z^{k+1}_{\beta}$ if $(k, \alpha) \searrow (k+1, \beta)$, and then remove some points if they  end up being too close to each other. Let us now do this in detail.
We define a probability $\mathbb{P}$, on the family of all relations $\searrow$ of the kind described, by setting
\begin{displaymath}
  \mathbb{P}((k, \alpha) \searrow (k+1, \beta)) = \frac{1}{\#\{\gamma: (k+1, \gamma) \le (k,\alpha)\}}
\end{displaymath}
for all $(k+1, \beta) \le (k, \alpha)$, and requiring that such events for two different $(k, \alpha)$ and $(\ell , \sigma)$ are independent. If $(k, \alpha) \searrow (k+1, \beta)$, we set
$y^k_{\alpha} = z^{k+1}_{\beta}$.  However, we want to now obtain better separation between the points $y^k_{\alpha}$. To this end, we say that $y^k_{\alpha}$ and $y^k_{\beta}$ are in conflict
if $\rho(y^k_{\alpha}, y^k_{\beta}) < (2A_0)^{-2}\vartheta^k$. If such is the case, we have $\rho(z^k_{\alpha}, z^k_{\beta}) < 3A_0^2\vartheta^k$. As $X$ is geometrically doubling, it follows that at most finitely many
pairs can conflict with a given pair $(k, \alpha)$. We enumerate the points $y^k_{\alpha}$ as $y^k_1, y^k_2, \ldots$. We choose $y^k_1$ and remove all the boundedly many points conflicting with it.
Next, choose the point with the smallest index in the remaining sequence, and remove all the boundedly many points conflicting with it. Continue this by induction. The final collection is now denoted by $x^k_{\alpha}$. By construction we have that
\begin{equation*}
   \rho(x^k_{\alpha}, x^k_{\beta}) \ge (2A_0)^{-2}\vartheta^k,\quad\text{if}\quad \alpha \ne \beta. 
\end{equation*}
Observe also that for an arbitrary $x \in X$ there exists $x^k_{\alpha}$ so that $\rho(x, x^k_{\alpha}) < 3A_0^2\vartheta^k$.

\begin{lem}\label{centerl}
If $(k+1, \beta) \le (k,\alpha)$, then $\mathbb{P}(z^{k+1}_{\beta} = x^k_{\alpha}) \ge \pi_0 > 0$.
\end{lem}

\begin{proof}
The event $z^{k+1}_{\beta} = x^k_{\alpha}$ requires that $(k, \alpha) \searrow (k+1, \beta)$ and then that the point
$y^k_{\alpha} = z^{k+1}_{\beta}$ was not removed in the above described removal process. As $X$ is geometrically doubling, we have that $\#\{\gamma: (k+1, \gamma) \le (k,\sigma)\} \lesssim 1$ and so always
$\mathbb{P}( (k, \sigma) \searrow (k+1, \gamma)) \ge \pi_0' > 0$ if $(k+1, \gamma) \le (k, \sigma)$. So all we need to prove is that for any of the $\lesssim 1$ pairs $(k, \sigma)$ with the potential of conflicting with $(k, \alpha)$, we have some $\gamma$ so that $(k+1, \gamma) \le (k, \sigma)$ and $\rho(z^{k+1}_{\beta}, z^{k+1}_{\gamma}) \ge (2A_0)^{-2}\vartheta^k$, for then there is a positive probability that $(k,\sigma)\searrow(k+1,\gamma)$ and no conflict with $(k,\alpha)$ will arise.

Consider any such $(k, \sigma)$. There is $z^{k+1}_{\gamma}$ so that $\rho(z^k_{\sigma}, z^{k+1}_{\gamma}) < \vartheta^{k+1}$. In particular, $\rho(z^k_{\sigma}, z^{k+1}_{\gamma}) < (2A_0)^{-1}\vartheta^k$ and so certainly
$(k+1, \gamma) \le (k, \sigma)$.
Note also that $(k+1, \beta) \le (k, \alpha) \ne (k, \sigma)$, and so $\rho(z^{k+1}_{\beta}, z^k_{\sigma}) \ge (2A_0)^{-1}\vartheta^k$. This yields that
\begin{displaymath}
\rho(z^{k+1}_{\beta}, z^{k+1}_{\gamma}) \ge A_0^{-1}\rho(z^{k+1}_{\beta}, z^k_{\sigma}) - \rho(z^{k+1}_{\gamma}, z^k_{\sigma}) \ge [2^{-1}A_0^{-2} - \vartheta]\vartheta^k > (2A_0)^{-2}\vartheta^k.
\end{displaymath}
This proves the assertion.
\end{proof}

Take a new random variable $\tau$, uniformly distributed on $[1,2]$ and independent of all the previous random quantities. Then define the random $\rho$-balls
\begin{displaymath}
B^k_{\alpha} = B(x^k_{\alpha}, \tau A_0^{-4}\vartheta^k/32).
\end{displaymath}
We now note that $\mathbb{P}(x \in \bigcup_{\alpha} B^k_{\alpha}) \ge \pi_0 > 0$ for all $x \in X$ and $k \in \Z$. Indeed, for a given $x \in X$ there exists
$z^{k+1}_{\beta}$ so that $\rho(x, z^{k+1}_{\beta}) < \vartheta^{k+1} < A_0^{-4}\vartheta^k/32$, and $\mathbb{P}(z^{k+1}_{\beta} = x^k_{\alpha}) \ge \pi_0$ as proved above.
We also have
\begin{displaymath}
\rho(B^k_{\alpha}, B^k_{\gamma}) \ge A_0^{-2}(2A_0)^{-2}\vartheta^k - 2^{-3}A_0^{-4}\vartheta^k = 2^{-3}A_0^{-4}\vartheta^k.
\end{displaymath}
So we have separation for balls of the same generation.

We now make the final construction of the balls. We are given some small $\upsilon \in (0,1)$ and a fixed starting size $k$.
We construct the level $k$ balls $B^k_{\alpha}$ as above. We take some small parameter $\omega \in (0,1)$ which we shall fix momentarily.
We introduced the random variable $\tau$ to make the proof of the following fact easy: it is unlikely for a point to belong to the set
$\bigcup_{\alpha} (1+\omega)B^k_{\alpha} \setminus B^k_{\alpha}$.

Let us spell this out. One notes that $x \in (1+\omega)B^k_{\alpha}$ can only happen for certain boundedly many different $\alpha$, where the bound depends on the geometric doubling property. We then estimate the probability that $x\in(1+\omega)B^k_{\alpha} \setminus B^k_{\alpha}$ for one of these balls $B^k_{\alpha}$. The mentioned inclusion happens if and only if
\begin{equation*}
  \frac{\tau\vartheta^k}{32A_0^4}\leq\rho(x,x^k_{\alpha}) < (1+\omega)\frac{\tau\vartheta^k}{32A_0^4}.
\end{equation*}
This means that $\tau$ must belong to a certain interval of length
\begin{equation*}
  32A_0^4\rho(x,x^k_{\alpha})\vartheta^{-k}\frac{\omega}{1+\omega}\leq 4\omega,
\end{equation*}
since necessarily $\rho(x,x^k_{\alpha})\leq(8A_0^4)^{-1}\vartheta^k$. Given the uniform distribution of $\tau$ on $[1,2]$, 
this implies that 
\begin{equation*}
  \mathbb{P}(x \in \bigcup_{\alpha} (1+\omega)B^k_{\alpha} \setminus B^k_{\alpha}) \le \eta_0 = \eta_0(\omega)\lesssim\omega.
\end{equation*}

We now choose $\omega$ so small, and then $M\in\N$ so large, that
\begin{displaymath}
  \frac{\pi_0}{\pi_0 + \eta_0} \ge 1 - \upsilon/2,\qquad
  1 - (1-\pi_0 - \eta_0)^M > \frac{1-\upsilon}{1-\upsilon/2}.
\end{displaymath}
Let further $\epsilon > 0$ be so small, and then $s\in\N$ so large, that
\begin{equation*}
  (1+\epsilon)^2 < 1 + \omega/3,\qquad 2A(\epsilon)\vartheta^s < \omega/3.
\end{equation*}
We now continue to make the above random ball covering with $k$ replaced by $k + s$ and $X$ replaced by $X \setminus \bigcup_{\alpha} (1+\omega)B^k_{\alpha}$.
We repeat this procedure $M$ times. We denote the collection of balls we obtain by $\mathcal{B}$.

We are in the following situation. At stage one a point belongs to some ball with probability $\pi_1 \ge \pi_0$ and to the $\omega$-buffer of some ball with probability $\eta_1 \le \eta_0$. Thus, a point belongs to none of these sets
with probability $1-\pi_1 - \eta_1$. Note that subsets of $X$ are geometrically doubling with the same constant $N$, and thus a point belongs to some ball at stage two with probability $\pi_2 \ge \pi_0$
and to the $\omega$-buffer of some ball with probability $\eta_2 \le \eta_0$. We have this situation at every stage. Therefore, it holds that
\begin{align*}
\mathbb{P}\big(x \in \bigcup_{B \in \mathcal{B}} B\big) &= \pi_1 + (1-\pi_1-\eta_1)\pi_2 + \cdots + \Big( \prod_{i=1}^{M-1} (1-\pi_i-\eta_i)\Big)\pi_M \\
 &\ge \pi_0 \sum_{i=0}^{M-1} (1-\pi_0-\eta_0)^i = \frac{\pi_0}{\pi_0 + \eta_0}[1 - (1-\pi_0 - \eta_0)^M] > 1 - \upsilon.
\end{align*}

We got $\rho$-balls of generation $k$, $k+s$, \ldots, $k + (M-1)s$ so that it is very likely for a point to belong to one of them. Also, balls of same generation $k + ms$ are $2^{-3}A_0^{-4}\vartheta^{k + ms}$-separated.
Now we need to utilize the regularity of the quasimetric $\rho$. Indeed, this is to guarantee that we can keep the buffer small but
still separate balls of different generations. Let us study two balls $B^{k+ms}$ and $B^{k+ns}$ of centers $x^{k+ms}$ and $x^{k+ns}$, where $m < n$. (We suppress the lower indices for this argument.) Choose $x \in B^{k+ms}$ and $y \in B^{k+ns}$.
Let $R$ and $r$ be the radii of $B^{k+ms}$ and $B^{k+ns}$, so that $r \le 2\vartheta^sR$. We have
\begin{align*}
(1+\omega)R &\le \rho(x^{k+ms}, x^{k+ns}) \\
&\le (1+\epsilon)\rho(x^{k+ms}, y) + A(\epsilon)\rho(y, x^{k+ns}) \\
&\le (1+\epsilon)^2\rho(x^{k+ms}, x) + (1+\epsilon)A(\epsilon)\rho(x, y) + A(\epsilon)\rho(y, x^{k+ns}) \\
&\le (1+\epsilon)^2R + (1+\epsilon)A(\epsilon)\rho(x, y) + 2A(\epsilon)\vartheta^sR \\
&\le R + (2\omega/3)R + (1+\epsilon)A(\epsilon)\rho(x,y)
\end{align*}
implying that $\rho(x,y) \ge c(\omega)R$.

Let us now formulate the above given construction of the random almost-covering by balls as a proposition.

\begin{prop}\label{rball}
Let $0 < \vartheta < A_0^{-4}/32$, $k \in \Z$ and $\upsilon \in (0,1)$ be given. Then there exist $\omega$, $s$ and $M$, independent of $k$, so that we may randomly construct a disjoint family  $\mathcal{B}$ of $\rho$-balls as follows:
if $B,B' \in \mathcal{B}$, then $r_B \sim \vartheta^{k+m_1s}$ and $r_{B'} \sim \vartheta^{k+m_2s}$ for some $m_1, m_2 = 0, 1, \ldots, M-1$ and
$\rho(B, B') \ge c(\omega)\vartheta^{k + \min(m_1,m_2)s}$, and also
\begin{displaymath}
\mathbb{P}\big(x \in \bigcup_{B \in \mathcal{B}} B\big) > 1- \upsilon
\end{displaymath}
for every $x \in X$.
\end{prop}

\section{Estimates for adjacent cubes of comparable size}

We are given adjacent ($d(Q, R) < CC_0\min(\ell(Q), \ell(R))$ say) $Q \in \mathcal{D}$ and $R \in \mathcal{D}'$ of comparable size, that is $\ell(Q) \sim \ell(R)$ (meaning $|\textrm{gen}(Q)-\textrm{gen}(R)| \le r$).
We are also given some fixed small $\epsilon > 0$.
We define $\Delta = Q \cap R$, $\delta_Q = \{x: d(x,Q) \le \epsilon \ell(Q) \textrm{ and } d(x, X \setminus Q) \le \epsilon \ell(Q)\}$
and $\delta_R = \{x: d(x, R) \le \epsilon \ell(R) \textrm{ and } d(x, X \setminus R) \le \epsilon \ell(R)\}$. Also, set
\begin{equation}\label{badr1}
Q_b = Q \cap \bigcup_{R' \in \mathcal{D}':\, \ell(R') \sim \ell(Q)} \delta_{R'}
\end{equation}
and
\begin{equation}\label{badr2}
R_b = R \cap \bigcup_{Q' \in \mathcal{D}:\, \ell(Q') \sim \ell(R)} \delta_{Q'}.
\end{equation}
We define $Q_s = Q \setminus \Delta \setminus \delta_R$, $Q_{\partial} = Q \setminus \Delta \setminus Q_s$,
$R_s = R \setminus \Delta \setminus \delta_Q$ and $R_{\partial} = R \setminus \Delta \setminus R_s$. Furthermore, set
$\tilde \Delta = \Delta \setminus \delta_Q \setminus \delta_R$. We now finally fix $\delta = A_0^{-4}/1000$, and then fix the smallest $k$ for which $\delta^k \le \Lambda^{-1}(8^{-1}\epsilon\min(\ell(Q), \ell(R)))^{\beta}$.
Consider some small enough $\upsilon \in (0,1)$, and set $\vartheta = \delta$. 
Recall Proposition \ref{rball},
that is, the random way to construct a collection of $\rho$-balls $\mathcal{B}$  starting from the fixed level $k$ with parameter $\upsilon$ (and the related parameters
$\omega$, $s$ and $M$ that all depend of $\upsilon$ but not on $k$).
As we have $\mathbb{P}(x \in \bigcup_{B \in \mathcal{B}} B) > 1 - \upsilon$ for all $x \in X$, we have $\mathbb{E}(\mu(\tilde \Delta \setminus \bigcup_{B \in \mathcal{B}} B)) < \upsilon\mu(\tilde \Delta)$.
So we may now fix some such ball covering $\mathcal{B}$ for which $\mu(\tilde \Delta \setminus \bigcup_{B \in \mathcal{B}} B) \le \upsilon\mu(\tilde \Delta)$ as we have positive probability to obtain one.
We now remove from the collection $\mathcal{B}$ those balls that do not touch $\tilde \Delta$.

First we want to estimate $\# \mathcal{B}$. Observe that $\textrm{diam}(\Delta) \le C_0\min(\ell(Q), \ell(R))$ and fix some $x_0 \in \Delta$.
If $B \in \mathcal{B}$, fix some $x \in B \cap \Delta$.
Denote the center of $B$ by $z_B$. Note that
\begin{align*}
d(z_B, x_0) &\le d(z_B, x) + d(x, x_0) \\
&\le 2\rho(z_B,x)^{1/\beta} + d(x,x_0) \\
&\le 2\delta^{k/\beta} + C_0\min(\ell(Q), \ell(R))  \\
&< [2 + C_0\Lambda^{1/\beta}\delta^{-1/\beta}8\epsilon^{-1}]\delta^{k/\beta} \\
&\le 16C_0\Lambda^{1/\beta}\delta^{-1/\beta}\epsilon^{-1}\delta^{k/\beta}.
\end{align*}
This means
that the $d$-ball $B(x_0, 16C_0\Lambda^{1/\beta}\delta^{-1/\beta}\epsilon^{-1}\delta^{k/\beta})$ contains the centers $z_B$ of the disjoint $d$-balls
$B(z_B, 4^{-1}(A_0^{-4}\delta^{(M-1)s}/32)^{1/\beta} \delta^{k/\beta})$. This implies that
\begin{displaymath}
\#\mathcal{B} \le N\Big(\frac{\epsilon}{64C_0}\Big(\frac{\delta^{(M-1)s+1}}{32\Lambda A_0^4}\Big)^{1/\beta}\Big)^{-n}.
\end{displaymath}
This is a dependence we can live with, as all the quantities in the upper bound will be eventually fixed.

Next, let us check that $\Lambda B \subset \Delta$ for every
$B \in \mathcal{B}$. There exists $x \in B$ so that $d(x, X \setminus \Delta) \ge \epsilon\min(\ell(Q), \ell(R))$ as $B \cap \tilde \Delta \ne \emptyset$. If $w \in \Lambda B$, we have
\begin{align*}
d(x, X \setminus \Delta) - d(w, X \setminus \Delta) &\le d(w,x) \\
&\le d(w,z_B) + d(z_B, x)  \\
&\le 2\rho(w,z_B)^{1/\beta} + 2\rho(z_B,x)^{1/\beta} \\
&\le 4\Lambda^{1/\beta}\delta^{k/\beta} \le \frac{\epsilon}{2}\min(\ell(Q), \ell(R)),
\end{align*}
and so $d(\Lambda B, X \setminus \Delta) \ge (1/2)\epsilon\min(\ell(Q), \ell(R)) > 0$.

In a forthcoming decomposition we shall have plenty of separated terms. For these the following lemma comes in handy, and we use it without further mention in what follows.

\begin{lem}
Let $S_1$ and $S_2$ be two sets so that we have $\operatorname{diam}(S_1) \sim \operatorname{diam}(S_2)$ and $d(S_1, S_2) \gtrsim \epsilon\min(\operatorname{diam}(S_1), \operatorname{diam}(S_2))$. Suppose we are also given functions
$\varphi$ and $\psi$ so that $\|\varphi\|_{L^{\infty}(\mu)} +  \|\psi\|_{L^{\infty}(\mu)} \lesssim 1$, spt$\, \varphi \subset S_1$ and spt$\,\psi \subset S_2$.
Then it holds that
\begin{displaymath}
  |\langle T\varphi, \psi \rangle| \lesssim \epsilon^{-d}\mu(S_1)^{1/2}\mu(S_2)^{1/2}.
\end{displaymath}
\end{lem}

\begin{proof}
Using the first kernel estimate we have
\begin{displaymath}
|\langle T\varphi, \psi \rangle| \lesssim \int_{S_2} \int_{S_1} \min\Big(\frac{1}{\lambda(x, d(x,y))}, \frac{1}{\lambda(y, d(x,y))}\Big)\,d\mu(y)\,d\mu(x).
\end{displaymath}
It holds that $\lambda(y,d(x,y)) \ge \lambda(y, d(S_1,S_2)) \gtrsim \lambda(y, \epsilon\operatorname{diam}(S_1))
\gtrsim \epsilon^{-d}\lambda(y, \operatorname{diam}(S_1)) \ge \epsilon^{-d}\mu(S_1)$. Similarly
$\lambda(x,d(x,y)) \gtrsim \epsilon^{-d}\mu(S_2)$. Thus, it follows that
\begin{displaymath}
|\langle T\varphi, \psi \rangle| \lesssim \epsilon^{-d}\min(\mu(S_1), \mu(S_2)) \le \epsilon^{-d}\mu(S_1)^{1/2}\mu(S_2)^{1/2}.
\end{displaymath}
\end{proof}

We write $\Delta \setminus \bigcup B$ as a disjoint union of $\Omega_i = \tilde \Delta \setminus \bigcup B$ and some
sets $\Omega_Q \subset Q_b$ and $\Omega_R \subset R_b$. We now decompose
\begin{eqnarray*}
\langle T(\chi_Q b_1), \chi_R b_2\rangle & = & \langle T(\chi_Q b_1), \chi_{R_{\partial}} b_2\rangle + \langle T(\chi_Q b_1), \chi_{R_s} b_2\rangle \\
& + & \langle T(\chi_{Q_{\partial}} b_1), \chi_{\Delta} b_2\rangle + \langle T(\chi_{Q_s} b_1), \chi_{\Delta} b_2\rangle \\
& + & \langle T(\chi_{\Delta} b_1), \chi_{\Delta \setminus \bigcup B} b_2\rangle + \langle T(\chi_{\Delta \setminus \bigcup B} b_1), \chi_{\bigcup B} b_2\rangle \\
& + & \langle T(\chi_{\bigcup B} b_1), \chi_{\bigcup B} b_2\rangle = A + B + C + D + E + F + G.
\end{eqnarray*}
Furhermore, we decompose
\begin{align*}
E &= \langle T(\chi_{\Delta} b_1), \chi_{\Delta \setminus \bigcup B} b_2\rangle \\
   &= \langle T(\chi_{\Delta} b_1), \chi_{\Omega_Q} b_2\rangle + \langle T(\chi_{\Delta} b_1), \chi_{\Omega_R} b_2\rangle
+ \langle T(\chi_{\Delta} b_1), \chi_{\Omega_i} b_2\rangle\\ &= E_1 + E_2 + E_3
\end{align*}
and
\begin{align*}
F &= \langle T(\chi_{\Delta \setminus \bigcup B} b_1), \chi_{\bigcup B} b_2\rangle \\
   &= \langle T(\chi_{\Omega_Q} b_1), \chi_{\bigcup B} b_2\rangle + \langle T(\chi_{\Omega_R} b_1), \chi_{\bigcup B} b_2\rangle
    +\langle T(\chi_{\Omega_i} b_1), \chi_{\bigcup B} b_2\rangle \\
    &= F_1 + F_2 + F_3.
\end{align*}
We still write
\begin{align*}
G &= \langle T(\chi_{\bigcup B} b_1), \chi_{\bigcup B} b_2\rangle \\
   &= \sum_B \langle T(\chi_B b_1), \chi_B b_2\rangle + \sum_{B \ne B'} \langle T(\chi_B b_1), \chi_{B'} b_2\rangle = G_1 + G_2.
\end{align*}

Let us deal with these terms now. We have for the terms
\begin{displaymath}
A = |\langle T(\chi_Q b_1), \chi_{R_{\partial}}b_2\rangle|, \,\,\, C = |\langle T(\chi_{Q_{\partial}} b_1), \chi_{\Delta} b_2\rangle|, \,\,\, E_1 = |\langle T(\chi_{\Delta} b_1), \chi_{\Omega_Q} b_2\rangle|
\end{displaymath}
and
\begin{displaymath}
E_2 = |\langle T(\chi_{\Delta} b_1), \chi_{\Omega_R} b_2\rangle|, \,\,\, F_1 = |\langle T(\chi_{\Omega_Q} b_1), \chi_{\bigcup B} b_2\rangle|, \,\,\, F_2 = |\langle T(\chi_{\Omega_R} b_1), \chi_{\bigcup B} b_2\rangle|
\end{displaymath}
that
\begin{displaymath}
C + E_1 + F_1 \lesssim \|T\| \|\chi_{Q_b}b_1\|_{L^2(\mu)}\mu(R)^{1/2}
\end{displaymath}
and
\begin{displaymath}
A + E_2 + F_2 \lesssim \|T\| \mu(Q)^{1/2}\|\chi_{R_b}b_2\|_{L^2(\mu)},
\end{displaymath}
where we have used the facts that $|b_1| \sim 1$ and $|b_2| \sim 1$. Next, we observe
that for the terms
\begin{displaymath}
E_3 = |\langle T(\chi_{\Delta} b_1), \chi_{\Omega_i} b_2\rangle| \qquad \textrm{and} \qquad F_3 = |\langle T(\chi_{\Omega_i} b_1), \chi_{\bigcup B} b_2\rangle|
\end{displaymath}
we have
\begin{displaymath}
E_3 + F_3 \lesssim \upsilon^{1/2} \|T\| \mu(Q)^{1/2}\mu(R)^{1/2}.
\end{displaymath}

For the separated terms
\begin{displaymath}
B = |\langle T(\chi_Q b_1), \chi_{R_s} b_2\rangle| \qquad \textrm{and} \qquad D =  |\langle T(\chi_{Q_s} b_1), \chi_{\Delta} b_2\rangle|
\end{displaymath}
we have the estimate
\begin{displaymath}
B + D \lesssim \epsilon^{-d} \mu(Q)^{1/2}\mu(R)^{1/2}.
\end{displaymath}
It remain to deal with the term $G$.
We invoke the weak boundedness property and the fact that $\Lambda B \subset \Delta$ for all the boundedly many $B \in \mathcal{B}$ to get that
\begin{displaymath}
G_1 = \Big|\sum_B \langle T(\chi_B b_1), \chi_B b_2\rangle\Big| \lesssim  C(\epsilon, \upsilon)\mu(Q)^{1/2}\mu(R)^{1/2}.
\end{displaymath}
Using the separation of different balls $B$ and $B'$ we obtain that
\begin{displaymath}
G_2 = \Big|\sum_{B \ne B'} \langle T(\chi_B b_1), \chi_{B'} b_2\rangle\Big| \lesssim C(\epsilon, \upsilon) \mu(Q)^{1/2}\mu(R)^{1/2}.
\end{displaymath}

We now recapitulate what we have done in form of a proposition.

\begin{prop}\label{ch9p}
Let $Q \in \mathcal{D}$ and $R \in \mathcal{D}'$ be two adjacent cubes of comparable size, that is, $d(Q, R) < CC_0\min(\ell(Q), \ell(R))$ and $|\textrm{gen}(Q)-\textrm{gen}(R)| \le r$.
Let $\epsilon > 0$ and $\upsilon \in (0,1)$. It holds that
\begin{align*}
|\langle T(\chi_Q b_1), \chi_R b_2\rangle| \lesssim \|T\| \|\chi_{Q_b}b_1&\|_{L^2(\mu)}\mu(R)^{1/2}+ \|T\| \mu(Q)^{1/2}\|\chi_{R_b}b_2\|_{L^2(\mu)} \\
&+ \upsilon^{1/2}\|T\| \mu(Q)^{1/2}\mu(R)^{1/2} + C(\epsilon, \upsilon)\mu(Q)^{1/2}\mu(R)^{1/2},
\end{align*}
where $Q_b$ and $R_b$ are as in (\ref{badr1}) and (\ref{badr2}) respectively.
\end{prop}

\section{Random dyadic systems}\label{sec:randomDyadic}

We now randomize our dyadic grids.
We first fix a reference system of dyadic points $(z^k_{\alpha})$ and the relation $\le$ essentially as in the case of the random ball covering (but working with the metric $d$ instead).
Indeed, for each $k \in \Z$ fix some maximal collection $z^k_{\alpha} \in X$ for which $d(z^k_{\alpha}, z^k_{\beta}) \ge \delta^k$ for all $\alpha \ne \beta$.
For each $(k, \alpha)$ there exists at least one $\beta$ for which $d(z^k_{\alpha},z^{k-1}_{\beta}) < \delta^{k-1}$. Also, there exists
at most one $\beta$ for which $d(z^k_{\alpha}, z^{k-1}_{\beta}) < \delta^{k-1}/2$. The ordering $\le$ is constructed using the rules we now describe. Consider any pair $(k, \alpha)$. Check first whether there exists $\beta$ so that
$d(z^k_{\alpha}, z^{k-1}_{\beta}) < \delta^{k-1}/2$. If so, set $(k, \alpha) \le (k-1,\beta)$ and $(k, \alpha) \not \le (k-1, \gamma)$ for $\gamma \ne \beta$. Otherwise, choose any $\beta$ for which
$d(z^k_{\alpha},z^{k-1}_{\beta}) < \delta^{k-1}$, and set $(k, \alpha) \le (k-1,\beta)$ and $(k, \alpha) \not \le (k-1, \gamma)$ for $\gamma \ne \beta$. Extend by transitivity.

Next, we introduce the transitive relation $\searrow$ exactly as before, and equipped with the same probabilistic notions. The new dyadic points $y^k_{\alpha}$ are build as before. The points are said to conflict if
$d(y^k_{\alpha}, y^k_{\beta}) < \delta^k/4$. Then we do the familiar removal procedure and get the final dyadic points $x^k_{\alpha}$, which satisfy
$d(x^k_{\alpha}, x^k_{\beta}) \ge \delta^k/4$ if $\alpha \ne \beta$. Observe also that for an arbitrary $x \in X$ there exists $x^k_{\alpha}$ so that $d(x, x^k_{\alpha}) < 3\delta^k$.
Using these one may then build a new relation $\le'$, similar to $\le$ but related to these new points, and the corresponding ``half-open'' cubes $Q^k_{\alpha}$.

Consider a given point $x \in X$. There exists $\beta$ so that $d(x, z^{k+1}_{\beta}) < \delta^{k+1} < \delta^k/500$. Let $(k+1, \beta) \le (k, \alpha)$.
We have $\mathbb{P}(z^{k+1}_{\beta} = x^k_{\alpha}) \ge \pi_0 > 0$ by an analog of Lemma \ref{centerl}. In particular, final dyadic points of consecutive generations $k$ and $k+1$ may well end up close to each other in this sense. Recalling
Lemma \ref{rlemma} this is relevant for the proof of the next lemma.

\begin{lem}\label{boundarylemma}
For some fixed $x \in X$ and $k \in \Z$, it holds
\begin{displaymath}
\mathbb{P}(x \in \delta_{Q^k_{\alpha}} \textrm{ for some }\alpha) \lesssim \epsilon^{\eta}
\end{displaymath}
for some $\eta > 0$.
\end{lem}

\begin{proof}
Recall the open and closed cubes $\tilde Q^k_{\alpha}$ and $\overline{Q}^k_{\alpha}$ and how they are related to the ``half-open'' cubes (we no longer use the hat notation so it may be a bit confusing).
One advantage of these is that they are determined by the centers a little bit differently than the ``half-open'' ones. Namely, to know these cubes for some generation $M$, it suffices to know the centers $x^{\ell}_{\beta}$ for generations $\ell \ge M$.

Fix the largest $m$ so that $500\epsilon \le \delta^m$. Now the point is to simply combine Lemma \ref{rlemma} with the last observation preceding this lemma.
Indeed, let the relation $\searrow$ be fixed from the level $k+m$ up. Choose some $\sigma_{k+m}$ so that $x \in \overline{Q}^{k+m}_{\sigma_{k+m}}$. We then randomly choose
the relation $\searrow$ between the levels $k+m$ and $k+m-1$. We have $\mathbb{P}(d(x^{k+m}_{\sigma_{k+m}}, x^{k+m-1}_{\beta}) < \delta^{k+m-1}/500 \textrm{ for some } \beta) \ge \pi_0$
so that $\mathbb{P}(d(x^{k+m}_{\sigma_{k+m}}, x^{k+m-1}_{\beta}) \ge \delta^{k+m-1}/500 \textrm{ for all } \beta) \le 1-\pi_0 =: \pi_1 < 1$. Let $(k+m, \sigma_{k+m}) \le' (k+m-1, \sigma_{k+m-1})$.
We again have $\mathbb{P}(d(x^{k+m-1}_{\sigma_{k+m-1}}, x^{k+m-2}_{\beta}) \ge \delta^{k+m-2}/500 \textrm{ for all } \beta) \le \pi_1 < 1$. We continue this way. Let $x \in \overline{Q}^k_{\alpha}$.
Lemma \ref{rlemma} implies together with independence that
\begin{align*}
\mathbb{P}(d(x, X \setminus \tilde Q^k_{\alpha}) < \epsilon\delta^k) &\le \mathbb{P}(d(x^{k+j}_{\sigma_{k+j}}, x^{k+j-1}_{\sigma_{k+j-1}}) \ge \delta^{k+j-1}/500 \textrm{ for all } j = 1, \ldots, m) \\
&\le \pi_1^m = (\delta^m)^{\log(\pi_1)/\log \delta} \lesssim \epsilon^{\eta},
\end{align*}
where $\eta = \log(\pi_1)/\log \delta > 0$. This was actually a conditional probability with the condition that the relation $\searrow$ was fixed in some way from the level $k+m$ up, but as this was arbitrary, the same estimate
holds without any conditionality. It remains to note that what we have done actually proves the whole lemma.
\end{proof}

\begin{thm}\label{bcubes}
For a fixed $Q^k_{\alpha}$ we have under the random choice of the other dyadic system that
\begin{displaymath}
\mathbb{P}(Q^k_{\alpha} \in \mathcal{D}_{\textrm{bad}}) \lesssim \delta^{r\gamma\eta}.
\end{displaymath}
\end{thm}

\begin{proof}
We make yet another assumption about the largeness of $r$. Namely,
we assume that $r$ is so large that $\delta^{r(1-\gamma)} < 1$ say. Let $R^{k-s}$ be the unique $\mathcal{D}'$-cube of generation $k-s$ containing
the center $x^k_{\alpha}$ of $Q^k_{\alpha}$. We have that
\begin{displaymath}
d(Q^k_{\alpha}, X \setminus R^{k-s}) \ge d(x^k_{\alpha}, X \setminus R^{k-s}) - C_0\delta^k.
\end{displaymath}
If $d(x^k_{\alpha}, X \setminus R^{k-s}) \ge 2\delta^{k\gamma}\delta^{(k-s)(1-\gamma)}$, then $d(Q^k_{\alpha}, X \setminus R^{k-s}) \ge \delta^{k\gamma}\delta^{(k-s)(1-\gamma)}$ for $s \ge r$
by the above inequality and the assumption that $\delta^{r(1-\gamma)} < 1$. Using a variant of the previous lemma we thus get that
\begin{displaymath}
\mathbb{P}(Q^k_{\alpha} \in \mathcal{D}_{\textrm{bad}}) \lesssim \sum_{s=r}^{\infty} (\delta^{\gamma \eta})^s \lesssim \delta^{r\gamma\eta}.
\end{displaymath}
\end{proof}

\section{Synthesis}

We now combine all these estimates to prove the non-trivial side of our main theorem, Theorem \ref{maintheorem}. Indeed, we now prove that $\|T\| \lesssim 1$. To this end,
choose $f$ and $g$ so that $|\langle Tf, g\rangle| \ge (1/2)\|T\|$, $\|f\|_{L^2(\mu)} = \|g\|_{L^2(\mu)} = 1$ and spt$\,f \subset B(x_0, \delta^m)$, spt$\,g \subset B(x_1, \delta^m)$ for some $x_0, x_1 \in X$ and $m \in \Z$. 
Write
\begin{displaymath}
f = \mathop{\sum_{Q \in \mathcal{D}}}_{\textrm{gen}(Q) \ge m} \Delta^{b_1}_Qf + \mathop{\sum_{Q \in \mathcal{D}}}_{\textrm{gen}(Q) = m} E^{b_1}_Q f = f_{\textrm{good}} + f_{\textrm{bad}},
\end{displaymath}
where
\begin{displaymath}
f_{\textrm{good}} = \mathop{\sum_{Q \in \mathcal{D}_{\textrm{good}}}}_{\textrm{gen}(Q) \ge m} \Delta^{b_1}_Qf + \mathop{\sum_{Q \in \mathcal{D}}}_{\textrm{gen}(Q) = m} E^{b_1}_Q f.
\end{displaymath}
We write the similar decomposition also for $g$, and then estimate
\begin{displaymath}
|\langle Tf, g\rangle| \le |\langle Tf_{\textrm{good}}, g_{\textrm{good}}\rangle| + |\langle Tf_{\textrm{good}}, g_{\textrm{bad}}\rangle| + |\langle Tf_{\textrm{bad}}, g\rangle|.
\end{displaymath}
Furthermore, we have
\begin{eqnarray*}
|\langle Tf_{\textrm{good}}, g_{\textrm{good}}\rangle| & \le & \sum |\langle T(\Delta^{b_1}_Qf), \Delta^{b_2}_Rg\rangle| + \sum |\langle T(\Delta^{b_1}_Qf), E^{b_2}_Rg\rangle| \\
& + & \sum |\langle T(E^{b_1}_Qf), \Delta^{b_2}_Rg\rangle| + \sum |\langle T(E^{b_1}_Qf), E^{b_2}_Rg\rangle|,
\end{eqnarray*}
where we sum over the obvious sets.

The first three series are similar so we only deal with the first one (by the above theory, it suffices that the term with the smaller support has zero integral). To this end, let us estimate the first series by (we agree that naturally all the time
$Q \in \mathcal{D}_{\textrm{good}}$ and $R \in \mathcal{D}'_{\textrm{good}}$)
\begin{displaymath}
\Big(\sum_{\ell(Q) \sim \ell(R)}  + \sum_{\ell(Q) \not \sim \ell(R): \, Q \cap R = \emptyset} + \sum_{\ell(Q) \not \sim \ell(R):\, Q \cap R \ne \emptyset}\Big) |\langle T(\Delta^{b_1}_Qf), \Delta^{b_2}_Rg\rangle|.
\end{displaymath}
The second series in the above decomposition is $\lesssim \|f\|_{L^2(\mu)}\|g\|_{L^2(\mu)} = 1$ by the sixth chapter (see Lemma \ref{ch6l} and Proposition \ref{ch6p}), while the third series is $\lesssim \|f\|_{L^2(\mu)}\|g\|_{L^2(\mu)} = 1$
by the seventh chapter (see (\ref{pares}), Theorem \ref{parabounded} and Proposition \ref{ch7p}).
We then write the first series in the above decomposition in the form
\begin{displaymath}
\mathop{\sum_{\ell(Q) \sim \ell(R)}}_{d(Q,R) \ge CC_0\min(\ell(Q), \ell(R))} |\langle T(\Delta^{b_1}_Qf), \Delta^{b_2}_Rg\rangle| + \mathop{\sum_{\ell(Q) \sim \ell(R)}}_{d(Q,R) < CC_0\min(\ell(Q), \ell(R))}
|\langle T(\Delta^{b_1}_Qf), \Delta^{b_2}_Rg\rangle|
\end{displaymath}
noting that the first series is $\lesssim \|f\|_{L^2(\mu)}\|g\|_{L^2(\mu)} = 1$ by the techniques used in the sixth chapter (see the beginning of the proof of Lemma \ref{ch6l} and Proposition \ref{ch6p}).
For a given cube $Q$ there exists only $\lesssim 1$ cubes $R$ such that
$\ell(Q) \sim \ell(R)$ and $d(Q,R) \lesssim \min(\ell(Q), \ell(R))$. Thus, we have by chapter 9 that (see Proposition \ref{ch9p})
\begin{eqnarray*}
\mathop{\sum_{\ell(Q) \sim \ell(R)}}_{d(Q,R) < CC_0\min(\ell(Q), \ell(R))} |\langle T(\Delta^{b_1}_Qf), \Delta^{b_2}_Rg\rangle| & \lesssim & \|T\|\Big(\sum_{\textrm{gen}(Q) > m} |A_Q|^2\|\chi_{Q_b}b_1\|_{L^2(\mu)}^2\Big)^{1/2} \\
& + & \|T\|\Big(\sum_{\textrm{gen}(R) > m} |B_R|^2\|\chi_{R_b}b_2\|_{L^2(\mu)}^2\Big)^{1/2} \\
& + & \upsilon^{1/2}\|T\| + C(\epsilon, \nu),
\end{eqnarray*}
where the constants $A_Q$ are related to the decomposition $\Delta_Q^{b_1}f = \sum A_{Q'}\chi_{Q'}b_1$, where we sum over the subcubes $Q'$ of $Q$. The constants $B_R$ are similarly related to $g$.

The fourth series involving the factors $E^{b_1}_Qf$ and $E^{b_2}_Rg$ has no terms with zero integral but the point is that there are only $\lesssim 1$ nonzero terms as the functions $f$ and $g$ are supported
on balls of radius $\delta^m$ and $\textrm{gen}(Q) = \textrm{gen}(R) = m$ in that sum. One can deal with the well separated terms using the first kernel estimate and use the estimates of chapter nine for the rest (see Proposition \ref{ch9p}).
The net result is that
\begin{eqnarray*}
(1/2)\|T\| \le |\langle Tf, g\rangle| & \lesssim & \|T\|\Big(\sum_{\textrm{gen}(Q) \ge m} |A_Q|^2\|\chi_{Q_b}b_1\|_{L^2(\mu)}^2\Big)^{1/2}\\
& + & \|T\|\Big(\sum_{\textrm{gen}(R) \ge m} |B_R|^2\|\chi_{R_b}b_2\|_{L^2(\mu)}^2\Big)^{1/2} \\
& + & \upsilon^{1/2}\|T\| + \|T\| \|g_{\textrm{bad}}\|_{L^2(\mu)} + \|T\| \|f_{\textrm{bad}}\|_{L^2(\mu)} + C(\epsilon, \nu),
\end{eqnarray*}
where the constants for the level $m$ come from writing $E_Q^{b_1}f =  A_Q\chi_Qb_1$ and $E_R^{b_2}g = B_R\chi_Rb_2$.
Choosing $r$ large enough, $\epsilon$ and $\upsilon$ small enough, and choosing the dyadic grids $\mathcal{D}$ and $\mathcal{D}'$ so that
that the first five terms (together with the implicit constants in front) contribute less than $(1/4)\|T\|$ yields that
$\|T\| \lesssim 1$ as desired. These details follow pretty much as in \cite{NTV} now that the lemmata in the previous chapter have been proven.

Let us quickly sketch the details for completeness. We can estimate $\mathbb{E}\|f_{\textrm{bad}}\|_{L^2(\mu)}^2$ (over the grids $\mathcal{D}$ and $\mathcal{D}'$) by introducing the square function
\begin{displaymath}
Sh(x) = \mathop{\sum_{Q \in \mathcal{D}}}_{\textrm{gen}(Q) \ge m} \|\Delta^{b_1}_Qh\|_{L^2(\mu)}^2\mu(Q)^{-1}\chi_Q + \mathop{\sum_{Q \in \mathcal{D}}}_{\textrm{gen}(Q) = m} \|E^{b_1}_Q h\|_{L^2(\mu)}^2\mu(Q)^{-1}\chi_Q.
\end{displaymath}
The point is that $\int_X Sh\,d\mu \sim \|h\|_{L^2(\mu)}^2$ for all $h$. By Theorem \ref{bcubes} we have for any fixed grid $\mathcal{D}$ taking the expectation over the grids $\mathcal{D}'$ that
\begin{displaymath}
\mathbb{E} \|f_{\textrm{bad}}\|_{L^2(\mu)}^2 \lesssim \int_X \mathbb{E} Sf_{\textrm{bad}}\,d\mu \lesssim \delta^{r\gamma\eta} \int_X Sf\,d\mu \lesssim \delta^{r\gamma\eta}\|f\|_{L^2(\mu)}^2 = \delta^{r\gamma\eta},
\end{displaymath}
and then the same holds if we take the expectation over all the grids $\mathcal{D}$ and $\mathcal{D}'$ too. We now once and for all fix $r$ to be so large that everything we have done above works and
that we have $\mathbb{E} \|f_{\textrm{bad}}\|_{L^2(\mu)}^2 \le 2^{-100}$. The same argument shows that $\mathbb{E} \|g_{\textrm{bad}}\|_{L^2(\mu)}^2 \le 2^{-100}$.

To deal with the remaining terms, write
\begin{displaymath}
f^k = \sum_{\textrm{gen}(Q) = k} A_Q\chi_Qb_1 \qquad \textrm{and} \qquad f^k_b = \sum_{\textrm{gen}(Q) = k} A_Q\chi_{Q_b}b_1.
\end{displaymath}
We have by Lemma \ref{boundarylemma} that $\mathbb{E}\chi_{Q_b}(x) \lesssim \epsilon^{\eta}\chi_Q(x)$ for all $x \in X$ and for all $Q$ in a fixed grid $\mathcal{D}$ (taking the expectation over the grids $\mathcal{D}'$), and thus
\begin{displaymath}
\mathbb{E} \sum_{k \ge m} \|f^k_b\|_{L^2(\mu)}^2 = \sum_{k \ge m} \int_X \mathbb{E} |f^k_b|^2\,d\mu \lesssim \epsilon^{\eta} \sum_{k \ge m} \|f^k\|_{L^2(\mu)}^2 \lesssim \epsilon^{\eta}\|f\|_{L^2(\mu)}^2 = \epsilon^{\eta}.
\end{displaymath}
The same then holds if we take the expectation over all the grids $\mathcal{D}$ and $\mathcal{D}'$ too. The same argument shows that also (with the obvious notations)
\begin{displaymath}
\mathbb{E} \sum_{k \ge m} \|g^k_b\|_{L^2(\mu)}^2 \lesssim \epsilon^{\eta}.
\end{displaymath}
This proves that we may choose the grids $\mathcal{D}$ and $\mathcal{D}'$ so that
\begin{displaymath}
(1/2-2^{-50}-C_4\upsilon^{1/2}-C_5\epsilon^{\eta/2})\|T\| \lesssim C(\epsilon, \upsilon),
\end{displaymath}
from which the claim follows by choosing the constants $\epsilon$ and $\upsilon$ properly.

\section{Application to Bergman-type operators}

Volberg and Wick \cite{VW} recently obtained a characterization of measures $\mu$ in the unit ball $\mathbb{B}_{2n}$ of $\C^n$ for which the analytic Besov--Sobolev space $B^{\sigma}_2(\mathbb{B}_{2n})$ embeds continuously into $L^2(\mu)$. (More precisely, they completed the picture by settling the remaining difficult case concerning $\sigma\in(1/2,n/2)$.) Their proof goes through a new $T1$ theorem for what they call ``Bergman-type'' operators. Let us describe the situation to see that this application (although, unfortunately, not their abstract $T1$ theorem behind it) could also be obtained as a consequence of our theory.

The measures $\mu$ in \cite{VW} satisfy the upper power bound $\mu(B(x,r))\leq r^m$, except possibly when $B(x,r)\subseteq H$, where $H$ is a fixed open set. However, in the exceptional case there holds $r\leq\delta(x):=d(x,H^c)$, and hence
\begin{equation*}
  \mu(B(x,r))\leq\lim_{\epsilon\to 0+}\mu(B(x,\delta(x)+\epsilon))
  \leq\lim_{\epsilon\to 0+}(\delta(x)+\epsilon)^m=\delta(x)^m.
\end{equation*}
Thus we find that their measures are actually upper doubling with
\begin{equation*}
  \mu(B(x,r))\leq\max(\delta(x)^m,r^m)=:\lambda(x,r).
\end{equation*}

The Bergman-type kernels $K(x,y)$ of \cite{VW} are required to have the pointwise estimate
\begin{equation*}
  |K(x,y)|\leq C\min\Big(\frac{1}{d(x,y)^m},\frac{1}{\max(\delta(x)^m,\delta(y)^m)}\Big),
\end{equation*}
where the upper bound is seen to be precisely the same
\begin{equation*}
  C\min\Big(\frac{1}{\lambda(x,d(x,y))},\frac{1}{\lambda(y,d(x,y))}\Big)
\end{equation*}
as required by our theory. However, the H\"older-continuity estimate is only assumed in the form
\begin{equation*}
  |K(x,y)-K(x',y)|+|K(y,x)-K(y,x')|\leq \frac{Cd(x,x')^{\alpha}}{d(x,y)^{m+\alpha}},
\end{equation*}
for $d(x,y)>2d(x,x')$, which is weaker than our condition when $\delta(x)\gg d(x,y)$. Hence the abstract main result of Volberg and Wick, \cite[Theorem~1]{VW}, is not as such included in our theory, but consists of a different extension of the non-homogeneous analysis of \cite{NTV}.

However, when it comes to the main application concerning the Besov--Sobolev spaces, \cite[Theorem~2]{VW}, the relevant kernel has the specific form
\begin{equation*}
  K(x,y)=(1-\bar{x}\cdot y)^{-m},\qquad x,y\in\bar{\mathbb{B}}_{2n}\subset\C^n.
\end{equation*}
Here $\bar{x}$ stands for the componentwise complex conjugation, and dot designates the usual dot product of $n$-vectors. Moreover, one equips $\bar{\mathbb{B}}_{2n}$ with the regular quasi-distance (see \cite[Lemma 2.6]{Tc})
\begin{equation*}
  d(x,y):=\big||x|-|y|\big|+\Big|1-\frac{\bar{x}\cdot y}{|x|\,|y|}\Big|.
\end{equation*}
Finally, the set $H$ related to the exceptional balls is now the open unit ball $\mathbb{B}_{2n}$. It is noteworthy that $\delta(x)=d(x,H^c)=1-|x|$ is the same as the distance of $x$ and $H^c$ in the Euclidean metric \cite[Lemma 2.8]{Tc}.

In \cite{VW} it is checked that this kernel $K$, the quasi-metric $d$, and the set $H$ indeed satisfy the Bergman-type kernel estimates. We now observe that even the standard estimates of our theory are verified. It is shown in \cite[Eq. (6)]{Tc} that $|1-\bar{x}\cdot y|\geq 3^{-1}d(x,y)$, and obviously $|1-\bar{x}\cdot y|\geq 1-|x|=\delta(x)$; similarly with $y$ in place of $x$. This confirms the first standard estimate, which we already knew. As for the H\"older-continuity, the proof of \cite[Proposition~2.13]{Tc} contains the bound
\begin{equation*}
  |K(x,y)-K(x',y)|\lesssim \Big(\frac{d(x,x')}{|1-\bar{x}\cdot y|}\Big)^{1/2}\frac{1}{|1-\bar{x}\cdot y|^m},
\end{equation*}
for $d(x,y)>Cd(x,x')$, and it suffices to use $|1-\bar{x}\cdot y|\gtrsim d(x,y)$ in the first factor and $|1-\bar{x}\cdot y|\gtrsim\max(d(x,y),\delta(x))$ in the second. The H\"older estimate with respect to the second variable is of course completely analogous.

In the formulation of their weak boundedness property and the BMO conditions, Volberg and Wick use certain ``cubes''  \cite[Sec.~7]{VW} which, just like in $\R^n$ with the usual distance, can actually be viewed as balls with respect to an equivalent regular quasimetric of $\ell^{\infty}$-type. So even this is compatible with our theory. Volberg and Wick conclude their paper \cite{VW} with essentially the same remark, with which Nazarov, Treil and Volberg started theirs \cite{NTV}, that ``these considerations can be extended to the case of metric spaces.'' And indeed they can!

\bibliographystyle{amsalpha}
\bibliography{metric}

\end{document}